\documentclass[10pt,reqno]{amsart}

\usepackage{dsfont}
\usepackage{amssymb, amsmath, amsthm}
\usepackage{enumerate,color}
\usepackage{graphicx}
\newtheorem{thm}{Theorem}[section]
\newtheorem{lem}[thm]{Lemma}
\newtheorem{cor}[thm]{Corollary}

\newtheorem{defn}[thm]{Definition}

\numberwithin{equation}{section}

\newcommand{\bel}{\begin{equation} \label}
\newcommand{\ee}{\end{equation}}
\def\beq{\begin{equation}}
\def\eeq{\end{equation}}
\newcommand{\bea}{\begin{eqnarray}}
\newcommand{\eea}{\end{eqnarray}}
\newcommand{\beas}{\begin{eqnarray*}}
\newcommand{\eeas}{\end{eqnarray*}}
\newcommand{\pd}{\partial}

\newcommand{\dd}{\mbox{d}}

\newcommand{\re}{\mathfrak R}

\newcommand{\R}{\mathbb{R}}
\newcommand{\C}{\mathbb{C}}

\allowdisplaybreaks

\def\epsilon{\varepsilon}
\def\phi {\varphi}

\providecommand{\abs}[1]{\left\lvert#1\right\rvert}
\providecommand{\norm}[1]{\left\lVert#1\right\rVert}

\renewcommand{\leq}{\leqslant}
\renewcommand{\geq}{\geqslant}

\providecommand{\abs}[1]{\left\lvert#1\right\rvert}
\providecommand{\norm}[1]{\left\lVert#1\right\rVert}

\title{\bf Simultaneous determination of coefficients,  internal sources and an obstacle  of a diffusion equation from a single measurement}

\author{Yavar Kian}
\address{Aix Marseille Univ, Universit\'e de Toulon, CNRS, CPT, Marseille, France.}
\email{yavar.kian@univ-amu.fr}

\begin{document}

\begin{abstract} This article is devoted to the simultaneous  resolution of three inverse problems, among the most important formulation of inverse problems for partial differential equations, stated for some class of diffusion equations from a single boundary measurement. Namely, we consider the simultaneous unique determination of several class of coefficients,  some internal sources  (a source term and an initial condition) and an obstacle appearing in a diffusion equation from a single boundary measurement. Our  problem can be formulated as the simultaneous determination of information about a diffusion process (velocity field, density of the medium), an obstacle and of the source of diffusion. We consider this problems in the context of a classical diffusion process described by a convection-diffusion equation as well as an anomalous diffusion phenomena  described by a time fractional diffusion equation.\\



\end{abstract}

\maketitle

\section{Introduction}

\subsection{Statement of the problem}

Let $\tilde{\Omega}$ and $\omega$ be two bounded  open set  of $\R^d$, $d \geq 2$, with $\mathcal C^{2}$ boundary, such that $\overline{\omega}\subset \tilde{\Omega}$ and such that $\Omega=\tilde{\Omega}\setminus \overline{\omega}$ is connected. Let $a \in \mathcal C^1(\overline{\Omega})$  satisfy the  condition
\bel{ell}
\exists c>0,\ a(x) \geq c ,\ x \in \Omega.
\ee
Fix  $q \in L^\infty(\Omega)$, such that 
\bel{a9}
q\geq 0  
\ee
and $B\in L^\infty(\Omega)^d$.
Given $T\in(0,+\infty)$, $\alpha\in(0,2)$  and 
$\rho \in L^\infty(\Omega)$, such that
\bel{eq-rho}
 0<\rho_0 \leq\rho(x) \leq\rho_M <+\infty,\quad x \in \Omega, 
\ee
we consider the initial boundary value problem (IBVP) 
\begin{equation}\label{eq1}
\begin{cases}
\rho(x)\partial_t^{\alpha}u -\textrm{div}\left(a(x) \nabla_x u \right)+B(x)\cdot\nabla_x u+q(x) u =  F, & \mbox{in }Q,\\
u= \Phi, & \mbox{on } (0,T)\times\partial\tilde{\Omega}, \\
 u=0,& \mbox{on }(0,T)\times\partial\omega,\\
\begin{cases}
u=u_0 & \mbox{if }0<\alpha\leq1,\\
u=u_0,\quad \partial_t u=0 & \mbox{if }1<\alpha<2,
\end{cases} & \mbox{in }\{0\}\times \Omega.
\end{cases}
\end{equation}
Here, we fix $Q=(0,T)\times\Omega$ and for $\alpha=1$ we denote by $\partial_t^\alpha$ the usual time derivative $\partial_t$ while, for $\alpha\in(0,1)\cup(1,2)$, $\partial_t^{\alpha}$ denotes the fractional Caputo derivative of order $\alpha$ with respect to $t$ defined by
\bel{cap} 
\pd_t^\alpha u(t,x):=\frac{1}{\Gamma([\alpha]+1-\alpha)}\int_0^t(t-s)^{[\alpha]-\alpha}\pd_s^{[\alpha]+1} u(s,x) \dd s,\ (t,x) \in Q.
\ee

Assuming that $\Phi\in W^{2,1}(0,T;H^{\frac{3}{2}}(\partial\tilde{\Omega}))$, $u_0\in L^2(\Omega)$, $F\in L^1(0,T;L^2(\Omega))$, it is well known that problem \eqref{eq1} admits a unique weak solution lying in $L^1(0,T;H^{2r}(\Omega))$, $r\in(0,1)$ (see e.g. \cite{JLLY,KY2,KSY,SY}).

Let $\Gamma_{in},\Gamma_{out}$ be two open subsets of $\partial\tilde{\Omega}$. In the present paper we study the inverse  problem of determining uniquely and simultaneously as much parameters as possible among the set  $\{a,\rho,B,q\}$ of coefficients, the set $\{u_0,F\}$ of internal source, the order of derivation in time $\alpha$ as well as the obstacle $\omega$ from a single boundary measurement  on the subset of the form $(T_0-\delta,T_0)\times\Gamma_{out}$, with $T_0\in(0,T)$ and $\delta\in(0,T_0]$, of the lateral boundary $(0,T)\times\partial\tilde{\Omega}$ for a suitable choice of the input $\Phi$ supported on $[0,T]\times\Gamma_{in}$.

\subsection{Motivations}

Let us mention that  diffusion equations  
of the form 
\eqref{eq1} describe       
diffusion of different kind of physical phenomena. While for $\alpha=1$ such equations correspond to convection-diffusion equations describing the transfer of different physical quantities (mass, energy, heat,...), for $\alpha\neq1$, equations of the form \eqref{eq1} are used for modeling different type of anomalous diffusion  process (diffusion  in inhomogeneous anisotropic porous media, turbulent plasma,  diffusion in a
turbulent flow,...). We refer to \cite{CSLG,JR,St} for more details about the applications of such equations.

The inverse problem addressed in the present paper corresponds to the simultaneous determination of a source of diffusion, an obstacle and of several parameters describing the  diffusion of some physical quantities. The convection term $B$ is associated 
with the velocity field of the moving quantities while the coefficients $(a,\rho,q)$ and the order of derivation $\alpha$  can be associated with some properties of the medium. Moreover, the source term $F$ and the initial condition $u_0$ can be seen as different kind of source of diffusion. For instance, our inverse problem can be stated as the determination of the velocity field and the density of the medium  as well as an obstacle and the source of diffusion of a contaminant in a soil from a single measurement at $\Gamma_{out}$. Moreover, for $\alpha\in(1,2)$, $\omega=\emptyset$, $F=0$, $B=q=0$ and $\rho=c^{\alpha}$, our inverse problem can be seen as the fractional formulation of the so called thermoacoustic tomography (TAT) and photoacoustic tomography (PAT), two coupled-physics process, used for combining the high resolution of ultrasound and the high contrast capabilities of electromagnetic waves, which can be formulated as the simultaneous determination of the wave speed and the initial pressure of a wave equation (see e.g. \cite{CLAB,KRK,LU,SU}).

\subsection{Known results}

Inverse problems for equations of the form \eqref{eq1} have received many attention these last decades. Many authors considered inverse coefficients, inverse source and inverse obstacle problems for  \eqref{eq1} when $\alpha=1$.  Without being exhaustive, we  mention the works of \cite{BK,CK1,CaK,CKY,Cho,ChK,ChY,EH,IK1,IK2,Katchalov2004,KR2}.  Contrary to $\alpha=1$, inverse problems associated with \eqref{eq1}  for $\alpha\in(0,1) \cup (1,2)$ has received more recent treatment. Most of these results correspond to inverse source problems (see e.g. \cite{FK,JLLY,KSXY,KY2,KJ}). For inverse coefficients problems, many  results have been stated with infinitely many measurements (see for instance \cite{KOSY,KSY,LIY}) among which the most general and precise results seem to be the ones stated in \cite{KOSY} where the measurements are restricted to a fixed time on a portion of the boundary of the domain. Several works have also been devoted to the recovery of coefficients form data given by final overdetermination (see e.g. \cite{KR1,KJ1}). To the best of our knowledge the works \cite{HLYZ,KLLY,KY2} are the only works in the mathematical literature where the recovery of coefficients appearing in fractional diffusion equations (in dimension higher than $2$) has been stated with  a single measurement which does not correspond to final overdetermination. Among these three works, \cite{KLLY} is the only  one with results stated with a single boundary measurement on a general bounded domain. Indeed, the result of \cite{HLYZ} is stated with internal measurement while the approach of \cite{KY2} is restricted to cylindrical domain $\Omega$. The approach of \cite{KLLY} is based on a generalization of the approach of \cite{AS,CY} (see also the work of \cite{E} for similar approach in the one dimensional case) based on the construction of a suitable Dirichlet input. Indeed, not only \cite{KLLY} extends the work of \cite{AS,CY} to fractional diffusion equations ($\alpha\neq1$) but it also extends the work of \cite{AS,CY} for $\alpha=1$ in terms of generality and precision. The main idea of \cite{KLLY} is to recover boundary data for a family of elliptic equations from a single boundary measurement of the solution of \eqref{eq1}, with $F=u_0\equiv0$ and $\omega=\emptyset$, and to combine this result with the works \cite{CK2,KKL,Katchalov2004,KKLO,LO,Po,Sa} in order to prove the recovery of coefficients appearing in \eqref{eq1}.

Let us observe that while, as mentioned above, several works have been devoted to the determination of space dependent coefficients or source terms, to the best of our knowledge, even for $\alpha=1$, there is no  result devoted to the simultaneous determination of space dependent internal source and coefficients appearing in problem \eqref{eq1} from single measurement. In the same way, we are not aware of any result devoted to the simultaneous determination of an obstacle and a coefficient or a source term appearing in \eqref{eq1} from a single measurement. Indeed, we have only find works  devoted to the simultaneous determination of source and coefficient, appearing in a parabolic equation, that depend only on the time variable (see e.g. \cite{Ka1,Ka2}). In the same way, we are only aware of the work of \cite{HKZ} for the simultaneous determination of a source term and an obstacle appearing in a hyperbolic equation.

\section{Statement of the main results}

Following \cite{AS,KLLY}, we start by introducing a suitable class of inputs $\Phi$. More precisely, we consider  $\chi\in \mathcal C^\infty(\partial\tilde{\Omega})$ such that supp$(\chi)\subset\Gamma_{in}$ and $\chi=1$ on $\Gamma_{in,*}$ an open subset of $\partial\tilde{\Omega}$. We fix $\tau_1,\tau_2\in (0,T]$, $\tau_1<\tau_2$, and a strictly increasing sequence $(t_k)_{k\geq 0}$ such that  $t_0=\tau_1$ and 
$\underset{k\to\infty}{\lim}t_k=\tau_2$. We fix also the sequence  $(c_k)_{k\geq 0}$ of $[0,+\infty)$ and we define  the sequence $(\psi_k)_{k\geq 1}$ of functions non-uniformly vanishing and lying in $\mathcal C^\infty(\R;[0,+\infty))$ defined, for all $k\in\mathbb N:=\{1,2,\ldots\}$, by
$$\psi_k(t)=\left\{\begin{aligned} 0\textrm{ \ for } t\in(-\infty,t_{2k-2}],\\ c_k\textrm{ \ for } t\in[t_{2k-1},+\infty).\end{aligned}\right.$$
We set the sequence  $(d_k)_{k\geq 1}$ of $(0,+\infty)$ such that 
$$\sum_{k=1}^\infty d_k\norm{\psi_k}_{W^{3,\infty}(\R_+)}<\infty.$$
In addition, we consider the sequence $(\eta_k)_{k\geq 1}$ of $H^{\frac{3}{2}}(\partial\tilde{\Omega})$ such that Span$(\{\eta_k:\ k\geq 1\})$ is dense in $H^{\frac{3}{2}}(\partial\tilde{\Omega})$ and $\norm{\eta_k}_{H^{\frac{3}{2}}(\partial\tilde{\Omega})}=1$, $k\in\mathbb N$. Finally, we define the input $\Phi\in \mathcal C^3([0,+\infty);H^{\frac{3}{2}}(\partial\tilde{\Omega}))$ as follows
\bel{g}\Phi(t,x):=\sum_{k=1}^\infty d_k\psi_k(t)\chi(x) \eta_k(x),\quad x\in\partial\tilde{\Omega},\ t\in[0,+\infty).\ee
It is  clear that supp$(\Phi)\subset[0,+\infty)\times\Gamma_{in}$.

Let us observe that according to \cite[Section 2.2]{KLLY} one can not expect more than the recovery of two coefficients among the set $\{a,\rho,B,q\}$. In the same way, following \cite[Section 1.3]{KSXY}, it is impossible to determine general time-dependent source terms from any kind of boundary measurements of the solution of \eqref{eq1}. For this purpose, in addition to the  two coefficients among the set $\{a,\rho,B,q\}$, we consider the recovery of the obstacle $\omega$, the order of derivation $\alpha$ and source terms of the form $F(t,x)=\sigma(t)f(x)$, with $\sigma$ a known function, and the recovery  of the initial condition $u_0$. 

For our  first main result, we consider this problem for $B\equiv0$ and a single boundary measurement  given by $a\partial_\nu u_{|(0,\tau_2)\times\Gamma_{out}}$, with $\nu$ the outward unit normal vector to $\partial\tilde{\Omega}$. This result can be stated as follows.

\begin{thm}
\label{t1} 
For $j=1,2$, let $\alpha_j\in(0,2)$, and let  the conditions
\bel{t1a}\Gamma_{in,*}\cup\Gamma_{out}=\partial\tilde{\Omega},\quad \Gamma_{in,*}\cap\Gamma_{out}\neq\emptyset,\ee
be fulfilled. We fix $\omega_j$, $j=1,2$, two open set of $\R^d$ with $\mathcal C^2$ boundary such that $\overline{\omega_j}\subset \tilde{\Omega}$ and such that  $\Omega_j=\tilde{\Omega}\setminus \overline{\omega_j}$  is connected. For $j=1,2$, we fix $(a_j,\rho_j,q_j)\in\mathcal C^1(\overline{\Omega_j})\times L^\infty(\Omega_j)\times L^\infty(\Omega_j)$ fulfilling \eqref{ell}-\eqref{eq-rho}, with $\Omega=\Omega_j$, and  we assume that either of the three following conditions:
$$(i)\ \rho_1=\rho_2\textrm{ on }\Omega_1\cap\Omega_2,\quad (ii)\ a_1=a_2\textrm{ on }\Omega_1\cap\Omega_2,\quad (iii)\ q_1=q_2\textrm{ on }\Omega_1\cap\Omega_2$$
and the conditions
\bel{t1b}
\nabla a_1(x)=\nabla a_2(x),\ x \in \partial\tilde{\Omega},
\ee
\bel{t1c}
\exists C>0,\ | \rho_1(x)-\rho_2(x) | \leq C \mathrm{dist}(x,\partial\tilde{\Omega})^2,\ x \in \Omega_1\cap\Omega_2,
\ee
are fulfilled. Moreover, for $j=1,2$, we fix   $u_0^j\in L^2(\Omega_j)$ and, for $\sigma\in L^1(0,T)$, $f_j\in L^2(\Omega_j)$, we define
\bel{source} F_j(t,x)=\sigma(t)f_j(x),\quad t\in(0,T),\ x\in\Omega_j.\ee
Here, we assume that the condition
\bel{source2} \textrm{supp}(\sigma)\subset [0,\tau_1)\ee
is fulfilled and we assume that the internal sources $u_0^j$, $f_j$, $j=1,2$, satisfy one of the following conditions 
$$(iv)\ f_1=f_2\textrm{ on }\Omega_1\cap\Omega_2,\quad (v)\ \sigma\not\equiv0,\ u_0^1=u_0^2\textrm{ on }\Omega_1\cap\Omega_2,$$
$$(vi)\ \sigma\not\equiv0,\ \textrm{there exists }\tau_0\in(0,\tau_1)\textrm{ such that supp}(\sigma)\subset(\tau_0,\tau_1).$$
Furthermore, we assume that  the expressions $\eta_1$ and $c_1$, appearing in the construction of the Dirichlet input $\Phi$ given by \eqref{g}, are such that $c_1=0$ and $\eta_1$ is a  function of constant sign lying in $W^{2-\frac{1}{r},r}(\partial\tilde{\Omega})$, for some $r>\frac{d}{2}$, and it satisfies $\chi\eta_1\not\equiv0$.  Finally, we assume that there exists a connected open subset $\tilde{O}$  of $\tilde{\Omega}\setminus (\overline{\omega_1\cup\omega_2})$ (see Figure \ref{fig1}) satisfying
\bel{ob1} \partial(\omega_1\cup\omega_2)\subset\partial\tilde{O},\quad \textrm{the interior of $\partial\tilde{O}\cap\Gamma_{out}$ is not empty},\ee
\bel{ob2} a_1(x)=a_2(x),\quad \rho_1(x)=\rho_2(x),\quad q_1(x)=q_2(x),\quad x\in\tilde{O}.\ee
Consider  $u^j$, $j=1,2$, the solution of \eqref{eq1} with $\Phi$ given by \eqref{g}, $\alpha=\alpha_j$, $\omega=\omega_j$, $B=0$, $(a,\rho,q)=(a_j,\rho_j,q_j)$ and $(u_0,F)=(u_0^j,F_j)$. Then the condition
\bel{t1d}a_1(x)\partial_\nu u^1(t,x)=a_2(x)\partial_\nu u^2(t,x),\quad (t,x)\in(0,\tau_2)\times \Gamma_{out}\ee
implies that \bel{t1aaa}\alpha_1=\alpha_2,\quad \omega_1=\omega_2,\quad a_1=a_2,\quad \rho_1=\rho_2,\quad q_1=q_2,\quad u_0^1=u_0^2,\quad f_1=f_2.\ee
\end{thm}

\begin{figure}[!ht]
  \centering
  \includegraphics[width=0.4\textwidth]{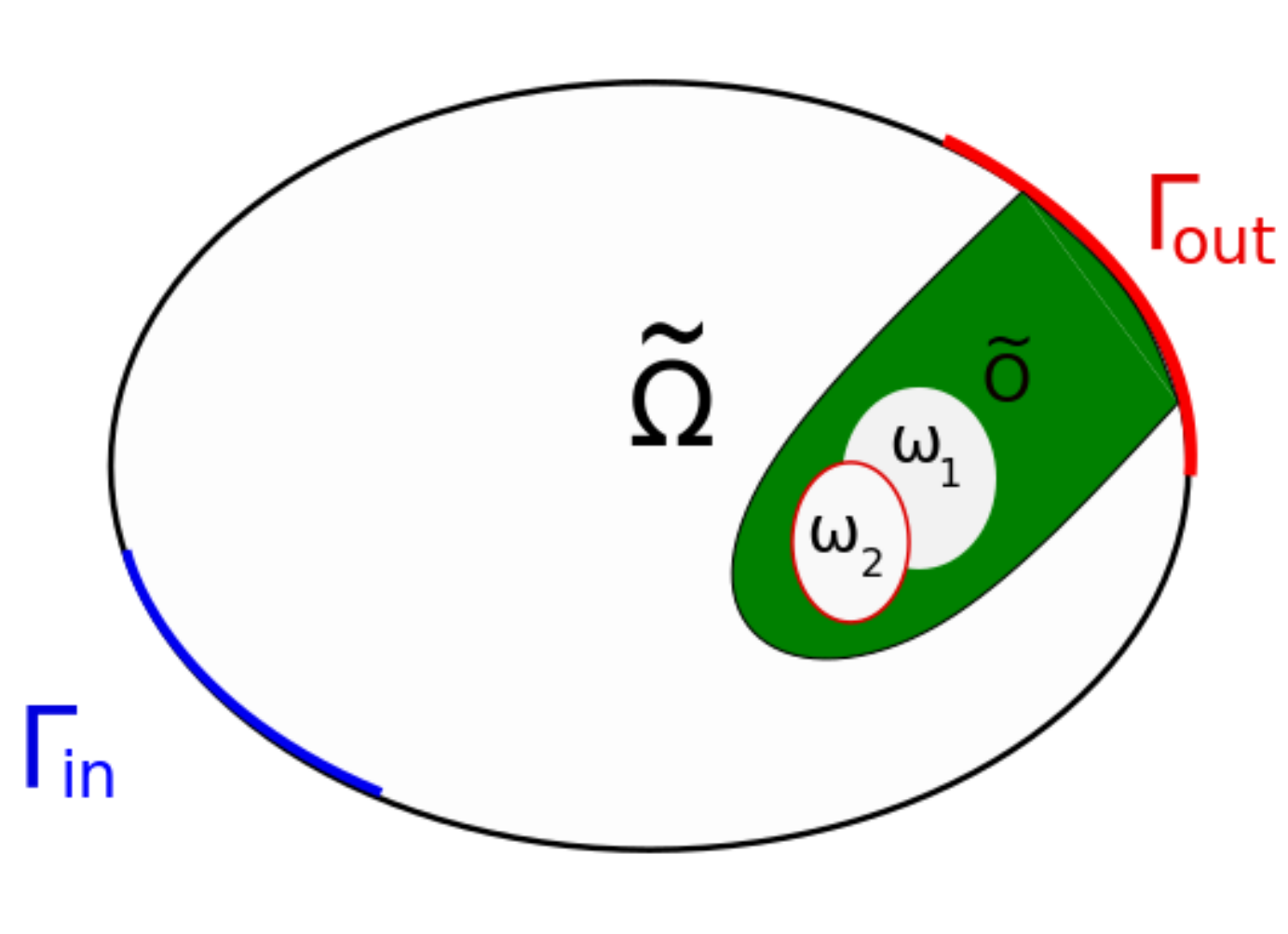}
  \caption{The sets $\tilde{\Omega}$, $\omega_1$, $\omega_2$ and $\tilde{O}$.  }
  \label{fig1}
\end{figure}

In the case $\alpha\in(0,2)\setminus\{1\}$, by considering some additional regularity assumptions, we can  extend the result of Theorem \ref{t1} into a result with a single measurement restricted to any time interval of the form $(T_0-\delta,T_0)$, with $T_0\in[\tau_2,T]$ and $\delta\in(0,T_0-\tau_1)$ arbitrary chosen. Namely, let us consider that $\tilde{\Omega}$ and $\omega$ are  $\mathcal C^4$, $\rho,a\in \mathcal C^3(\overline{\Omega})$, $q\in W^{2,\infty}(\Omega)$, $f\in H^2(\Omega)$, $u_0\in H^{\lceil\alpha\rceil}_0(\Omega)\cap H^{2\lceil\alpha\rceil}(\Omega)$, where  $\lceil\cdot\rceil$ denotes the ceiling function and $H^k_0(\Omega)$, $k\in\mathbb N$, denotes the closure of $\mathcal C^\infty_0(\Omega)$ in $H^k(\Omega)$.
Let also  $\sigma\in W^{3,1}(0,T)$ be such that $\sigma(0)=\sigma'(0)=\sigma^{(2)}(0)=0$. We suppose that  the Dirichlet input $\Phi$, given by \eqref{g}, is defined with $(\eta_k)_{k\geq 1}$ a sequence of functions of $H^{\frac{7}{2}}(\partial\tilde{\Omega})$ such that Span$(\{\eta_k:\ k\geq 1\})$ is dense in $H^{\frac{3}{2}}(\partial\tilde{\Omega})$ and $\norm{\eta_k}_{H^{\frac{7}{2}}(\partial\tilde{\Omega})}=1$, $k\in\mathbb N$. We recall that with this choice of the functions  $(\eta_k)_{k\geq 1}$, we have  $\Phi\in \mathcal C^3([0,T];H^{\frac{7}{2}}(\partial\tilde{\Omega}))$ with
$$\Phi(0,x)=\partial_t\Phi(0,x)=\partial_t^2\Phi(0,x)=0,\quad x\in\partial\tilde{\Omega}.$$
Combining \cite[Proposition 2.6, 2.8]{KY3} with \cite[Theorem 2.5, 2.9]{KY3}, one can check that problem \eqref{eq1} admits a unique solution $u\in  W^{\lceil\alpha\rceil,1}(0,T;H^{\frac{7}{4}}(\Omega))\cap L^1(0,T;H^{2+\frac{7}{4}}(\Omega))$. In particular we have $\partial_\nu u\in W^{\lceil\alpha\rceil,1}(0,T;L^2(\partial\tilde{\Omega}))$ and $\partial_\nu \textrm{div}(a\nabla u)\in L^1(0,T;L^2(\partial\tilde{\Omega}))$. Under these smoothness assumptions, we can prove that the result of Theorem \ref{t1}, with $\alpha$ known, remains valid with measurement given by $\partial_\nu u$  and $\partial_\nu \textrm{div}(a\nabla u)$ restricted to $(T_0-\delta,T_0)\times\Gamma_{out}$, with $T_0\in[\tau_2,T]$ and $\delta\in(0,T_0-\tau_1)$ arbitrary chosen. Our result for this problem can be stated as follows.

\begin{thm}
\label{tt2} 
Let $\alpha_1=\alpha_2=\alpha\in(0,2)\setminus\{1\}$  and let  the conditions of Theorem \ref{t1} be fulfilled.
 We assume also that, for $j=1,2$,   $\rho_j,a_j\in \mathcal C^3(\overline{\Omega_j})$, $q_j\in W^{2,\infty}(\Omega_j)$, $u_0^j\in H^{\lceil\alpha\rceil}_0(\Omega_j)\cap H^{2\lceil\alpha\rceil}(\Omega_j)$, $f_j\in H^2(\Omega_j)$, satisfy the condition of Theorem \ref{t1} as well as the following conditions
\bel{t6d}\partial_\nu \textrm{\emph{div}}(a_1(x)\nabla h)(x)=\partial_\nu \textrm{\emph{div}}(a_2(x)\nabla h)(x),\quad  h\in H^4(\tilde{\Omega}),\ x\in\Gamma_{out},\ee
\bel{t6e}  a_1(x)=a_2(x),\quad \partial_\nu^k \rho_1(x)=\partial_\nu^k \rho_2(x),\quad \partial_\nu^k q_1(x)=\partial_\nu^k q_2(x),\quad k=0,1,\  x\in\Gamma_{out}.\ee
Moreover, we assume that $\tilde{\Omega}$ and $\omega_j$, $j=1,2$,  are $\mathcal C^4$, $\sigma\in W^{3,1}(0,T)$, with $\sigma(0)=\sigma'(0)=\sigma^{(2)}(0)=0$,  and the Dirichlet input \eqref{g} is defined with $(\eta_k)_{k\geq 1}$ a sequence of functions lying in $H^{\frac{7}{2}}(\partial\tilde{\Omega})$ such that Span$(\{\eta_k:\ k\geq 1\})$ is dense in $H^{\frac{3}{2}}(\partial\tilde{\Omega})$ and $\norm{\eta_k}_{H^{\frac{7}{2}}(\partial\tilde{\Omega})}=1$, $k\in\mathbb N$.
 Consider  $u^j$, $j=1,2$, the solution of \eqref{eq1} with $\Phi$ given by \eqref{g}, $B=0$, $(a,\rho,q)=(a_j,\rho_j,q_j)$ and $(u_0,F)=(u_0^j,F_j)$. Then, for any arbitrary chosen $T_0\in[\tau_2,T]$ and $\delta\in(0,T_0-\tau_1)$, the condition
\bel{t6a}\left\{\begin{aligned}&\partial_\nu u^1(t,x)=\partial_\nu u^2(t,x),\\ &\partial_\nu \textrm{\emph{div}}\left(a_1(x) \nabla_x u^1\right) (t,x)=\partial_\nu \textrm{\emph{div}}\left(a_1(x) \nabla_x u^2\right) (t,x),\end{aligned}\right. \quad (t,x)\in (T_0-\delta,T_0)\times\Gamma_{out}\ee
implies that \eqref{t1aaa} holds true.
\end{thm}

For our third main result, we consider the above problem for $a=1$, $q=0$ and for $\Gamma_{in,*}=\Gamma_{out}=\partial\tilde{\Omega}$.  Our third main result can be stated as follows.

\begin{thm}
\label{t2}                                                                   
Let  $d\geq3$ and, for $j=1,2$, let $\alpha_j\in(0,1]$ $a_j=1$, $q_j=0$, $\omega_j$ be an open set of $\R^d$ with $\mathcal C^2$ boundary such that $\overline{\omega_j}\subset \tilde{\Omega}$ and such that $\Omega_j=\tilde{\Omega}\setminus \overline{\omega_j}$ is connected, $\rho_j\in \mathcal C(\overline{\Omega_j})$ satisfy \eqref{eq-rho}, with $\Omega=\Omega_j$, and let $B_j\in\mathcal C^\gamma(\overline{\Omega_j})^d$, with $\gamma\in(2/3,1)$. Moreover, for $j=1,2$, we fix   $u_0^j\in L^2(\Omega_j)$, $\sigma\in L^1(0,T)$, $f_j\in L^2(\Omega_j)$, satisfying \eqref{source2} and one of the conditions (iv), (v), (vi), and we consider $F_j$ given by \eqref{source}. We assume also that the expression $\eta_1$ appearing in the construction of the Dirichlet input $\Phi$, given by \eqref{g}, is lying in $W^{2-\frac{1}{r},r}(\partial\tilde{\Omega})$, for some $r>\frac{d}{2}$, and $\chi\eta_1\not\equiv0$. Finally, we assume that there exists a connected open subset $\tilde{O}$ of $\tilde{\Omega}\setminus (\overline{\omega_1\cup\omega_2})$ satisfying \eqref{ob1} with $\Gamma_{out}=\partial\tilde{\Omega}$ such that $(\R^3\setminus\tilde{\Omega})\cup \overline{\tilde{O}}\cup\omega_1$ is connected and the following conditions
\bel{ob3} B_1(x)=B_2(x),\quad \rho_1(x)=\rho_2(x),\quad x\in\tilde{O},\ee
\bel{ob4} \alpha_1=\alpha_2\quad \textrm{or}\quad \omega_1=\omega_2=\emptyset\ee
are fulfilled.  Consider  $u^{j}$, $j=1,2$, the solution of \eqref{eq1} with $(a,\alpha,B,\rho,q)=(a_j,\alpha_j,B_j,\rho_j,q_j)$, $(u_0,F)=(u_0^j,F_j)$ and $\Phi$ given by \eqref{g} with $\chi=1$. 
Then the condition
\bel{t2a}\partial_\nu u^{1}(t,x)=\partial_\nu u^{2}(t,x),\quad (t,x)\in(0,\tau_2)\times \partial\tilde{\Omega}\ee
implies that 
\bel{t2aa}\omega_1=\omega_2,\quad \alpha_1=\alpha_2,\quad B_1=B_2,\quad \rho_1=\rho_2,\quad u_0^1=u_0^2,\quad f_1=f_2.\ee  \end{thm}

Let us recall that our construction of the input $\Phi$ given by \eqref{g} can also be extended to $(M,g)$ a compact connected and smooth Riemanian manifold with boundary by replacing $\partial\tilde{\Omega}$ with $\partial M$. In that case, we define  the  Laplace-Beltrami operator
\[
\Delta_{g}:=\textrm{div}_g (\nabla_g\cdot)
\]
where $\textrm{div}_g$ and $\nabla_g$ denote divergence and gradient operators on $(M,g)$ respectively, and we consider  the following problem  on the manifold $M$

\begin{equation}\label{eqM1}
\begin{cases}
\partial_t^{\alpha}u -\Delta_{g} u +q(x) u =  F, & \mbox{in }(0,T)\times M,\\
u= \Phi, & \mbox{on } (0,T)\times\partial M, \\
\begin{cases}
u=u_0 & \mbox{if }0<\alpha\leq1,\\
u=u_0,\quad \partial_t u=0 & \mbox{if }1<\alpha<2,
\end{cases} & \mbox{in }\{0\}\times M.
\end{cases}
\end{equation}
In that case, following the results of \cite{KLLY} based on the works of \cite{KKL,KKLO,KOSY,LO} we obtain the following extension of our results to Riemanian manifolds.

\begin{cor}\label{c1}
For $j=1,2$, let $\alpha_j\in(0,2)$, let $(M_j,g_j)$ be two compact and smooth connected Riemannian manifolds of dimension $d\ge2$ with the same boundary and with $g_1=g_2$ on $\partial M_1$, and let  $q_j\in C^\infty(M_j)$ satisfy  $q_j\ge0$ on $M_j$. Moreover, for $j=1,2$, we fix   $u_0^j\in L^2(M_j)$, $\sigma\in L^1(0,T)$, $f_j\in L^2(M_j)$ satisfying \eqref{source2} and one of the conditions 
$$(iv')\ f_j=0,\ j=1,2,\quad (v')\ \sigma\not\equiv 0,\  u_0^j\equiv0,\  j=1,2,$$
$$(vi')\ \sigma\not\equiv0,\ \textrm{there exists }\tau_0\in(0,\tau_1)\textrm{ such that supp}(\sigma)\subset(\tau_0,\tau_1).$$
Furthermore, we assume that  the expressions $\eta_1$ and $c_1$, appearing in the construction of the Dirichlet input $\Phi$ given by \eqref{g} with $\partial\tilde{\Omega}$ replaced by $\partial M_1$, are such that $c_1=0$ and $\eta_1$ is a  function of constant sign lying in $\mathcal C^3(\partial M_1)$ and it satisfies $\chi\eta_1\not\equiv0$. 
We consider also $F_j$ given by \eqref{source} and we fix $\Gamma_{in}=\Gamma_{out}$ an arbitrary open subset of $\partial M_1$.
Consider $u^j$, $j=1,2$,  the solution of \eqref{eqM1} with $\Phi$ given by \eqref{g}, $\alpha=\alpha_j$, $(M,g)= (M_j,g_j)$, $q=q_j$, $u_0=u_0^j$, $F=F_j$. Then the condition
\begin{equation}\label{t4d}
\partial_{\nu}u^1=\partial_{\nu}u^2\quad\mbox{in }(0,\tau_2)\times\Gamma_{out}
\end{equation}
implies that $(M_1,g_1)$ and $(M_2,g_2)$ are isometric. Moreover, \eqref{t4d} implies that there exist $\phi\in C^\infty(M_2;M_1)$, an isomtery from $(M_2,g_2)$ to $(M_1,g_1)$, fixing $\partial M_1$ and depending only on $(M_j,g_j)$, $j=1,2$, such that
\begin{equation}\label{mu_c}
 \alpha_1=\alpha_2,\quad q_2=q_1\circ\phi,\quad u_0^2=u_0^1\circ\phi,\quad  f_2=f_1\circ\phi.
\end{equation}
\end{cor}
In the spirit of Theorem \ref{tt2}  we can also restrict the measurement under consideration in Corollary \ref{c1} to $(T_0-\delta,T_0)\times\Gamma_{out}$. This result can be stated as follows.

\begin{cor}\label{cc1}
We assume that the conditions of Corollary \ref{c1} are fulfilled.
Let $\alpha\in(0,2)\setminus\{1\}$, let $(M,g_j)$ be two compact and smooth connected Riemannian manifolds of dimension $d\ge2$ with $g_1=g_2$ on $\partial M$, and let  $q_j\in C^\infty(M)$ satisfy  $q_j\ge0$ on $M$. Moreover, for $j=1,2$, we assume that   $u_0^j \in H^{\lceil\alpha\rceil}_0(M)\cap H^{2\lceil\alpha\rceil}(M)$, $f_j\in H^2(M)$ and  $\sigma\in W^{3,1}(0,T)$ satisfies $\sigma(0)=\sigma'(0)=\sigma^{(2)}(0)=0$.
 We consider also $F_j$ given by \eqref{source} and  the Dirichlet input \eqref{g} is defined with $(\eta_k)_{k\geq 1}$ a sequence of functions of $H^{\frac{7}{2}}(\partial M)$ such that Span$(\{\eta_k:\ k\geq 1\})$ is dense in $H^{\frac{3}{2}}(\partial M)$ and $\norm{\eta_k}_{H^{\frac{7}{2}}(\partial M)}=1$, $k\in\mathbb N$. Assume also that $\Gamma_{in}=\Gamma_{out}$ is arbitrary chosen and the conditions 
\bel{cc1c}\partial_\nu\Delta_{g_1}h(x)=\partial_\nu\Delta_{g_2}h(x),\quad x\in\Gamma_{out},\ h\in H^4(M),\ee
\bel{cc1d}\partial_\nu^k c^1(x)=\partial_\nu^k c^2(x),\quad k=0,1,\  x\in\Gamma_{out}\ee
are fulfilled. Consider $u^j$, $j=1,2$,  the solution of \eqref{eqM1}   and $(g,q,u_0,F)=(g_j,q_j, u_0^j,F_j)$. Then, for any arbitrary chosen $T_0\in[\tau_2,T]$ and $\delta\in(0,T_0-\tau_1)$, the conditions
\bel{cc1a}
\partial_{\nu}u^1=\partial_{\nu}u^2\quad\mbox{in }(T_0-\delta,T_0)\times\Gamma_{out}
\ee
\bel{cc1b}
\partial_{\nu}\Delta_{g_1}u^1=\partial_{\nu}\Delta_{g_1}u^2\quad\mbox{in }(T_0-\delta,T_0)\times\Gamma_{out}\ee
imply that $(M,g_1)$ and $(M,g_2)$ are isometric and \eqref{mu_c} is fulfilled with $M_1=M_2=M$.
\end{cor}

We can also extend our results to the case where the input and the measurements are applied on some disjoint sets with respect to the space variable.

\begin{cor}\label{c2}
Let the condition of Corollary \ref{c1} be fulfilled and denote by $u^j$, $j=1,2$, the solution of \eqref{eqM1} with $\Phi$ given by \eqref{g}, $(M,g,u_0,F)=(M_j,g_j, u_0^j,F_j)$, $j=1,2$, and $q\equiv0$. In addition, we assume that the wave equation on $(0,+\infty)\times M_j$, $j=1,2$, is exactly controllable from $\Gamma_{in,*}$$^1$\footnote{$^1$ Here we refer to \cite{Bardos1992} for geometrical conditions that guarantee the exact controllability of the wave equation from $\Gamma_{in,*}$. 
} and $\overline{\Gamma_{in}}\cap\overline{\Gamma_{out}}=\emptyset$. Then \eqref{t4d} implies that $(M_1,g_1)$ and $(M_2,g_2)$ are isometric and \eqref{mu_c} holds true.
\end{cor}

Let us observe that the results of Theorem \ref{t1}, \ref{tt2} and \ref{t2} correspond to the simultaneous unique determination of two coefficients among the set of parameters $\{\rho,a,B,q\}$, the order of derivation $\alpha$, the obstacle $\omega$ and the  internal space dependent sources  $\{u_0, f\}$, from a single boundary measurement of the solution of \eqref{eq1}.  In the same way, Corollary \ref{c1}, \ref{c2} provide the simultaneous unique determination (up to isometry) of the Riemannian manifold  $(M,g)$ as well as  the internal sources $\{u_0, f\}$. Assuming that $\omega_1=\omega_2=\emptyset$, the results of Theorem \ref{t1}, \ref{tt2} and \ref{t2} can be stated in terms of simultaneous recovery of the set of the coefficients, the order of derivation in time and the  internal source without requiring the extra conditions \eqref{ob1}-\eqref{ob2} and \eqref{ob3}-\eqref{ob4}.

To the best of our knowledge, even for $\alpha=1$, the results of Theorem \ref{t1} and \ref{t2} correspond to the first resolution of three among the most important class of inverse problems (inverse coefficient, inverse source and inverse obstacle problems) stated for partial differential equations from a single boundary measurement. In addition in Corollary \ref{c1}, \ref{c2}, we extend, for what seems to be the first time, this approach to the simultaneous determination of a Riemannian manifold (up to isometry) and an internal source. While several authors considered the recovery of coefficients appearing in different evolution PDEs from a single boundary measurement (e.g. \cite{FeK,KLLY,Ya}) only some restricted results deal with the simultaneous determination of space dependent coefficients and internal source appearing in an evolution PDE from a single boundary measurement (see \cite{LU}) and none of them seems to consider the simultaneous determination of source, obstacle and coefficients from a single boundary measurement.  In that sense, the results of Theorem \ref{t1}, \ref{t2} and Corollary \ref{c1}, \ref{c2} correspond to, what seems to be, the first results of  simultaneous determination  of general class of space dependent coefficients and internal source and the first results of simultaneous determination of source, obstacle and coefficients from a single boundary measurement for all evolution linear PDEs.

The approach that we use in this work is inspired by the analysis of \cite{KLLY}  mainly devoted to the  determination of coefficients appearing in \eqref{eq1} from a single boundary measurement. In the present paper we extend the work of \cite{KLLY} in five different directions: 1) We give in Theorem \ref{t5} a simplified proof of \cite[Proposition 3.2]{KLLY} based on properties of  time analiticity  of some solutions of problem \eqref{eq1}; 2) We prove that with a class of Dirichlet input similar to the one considered by \cite{KLLY}, in addition to the recovery of coefficients, one can also prove the recovery of internal sources; 3) We add to the recovery of coefficients considered by \cite{KLLY}, the recovery of an obstacle from similar data; 4) In Theorem \ref{t1} and in Corollary \ref{c1}, \ref{c2}, we add to the results with partial data and results stated on manifold of \cite{KLLY} (see \cite[Theorem 2.2]{KLLY} and \cite[Corollary 2.4, 2.6, 2.7]{KLLY}) the recovery of the order of derivation $\alpha$;   5) For $\alpha\neq1$, we show in Theorem \ref{tt2} and  Corollary \ref{cc1}, how one can restrict the measurement considered in \cite[Theorem 2.2]{KLLY} and \cite[Corollary 2.4, 2.6, 2.7]{KLLY} to any time interval of the form $(T_0-\delta,T_0)$ where $\delta$ can be arbitrary small.

Let us remark that, for $\alpha\neq1$, in Theorem \ref{tt2} and in Corollary \ref{cc1} we can restrict our single measurement to any interval in time of the form $(T_0-\delta,T_0)$, with $T_0\in[\tau_2,T]$ and $\delta\in(0,T_0-\tau_1)$ arbitrary chosen, while all other comparable results that we know consider measurement on an interval in time of the form $(0,T_0)$. This improvement of the known results can be applied to the important and difficult problem of  determining coefficients of a PDE from excitation and single measurement made on disjoint sets of the lateral boundary $(0,T)\times\partial\tilde{\Omega}$ (resp. $(0,T)\times\partial M$). Indeed,  assuming that $c_k=0$, $k\in\mathbb N$, with $c_k$ introduced in the definition of the Dirichlet input $\Phi$, and choosing $T_0\in(\tau_2,T)$, $\delta\in(0,T_0-\tau_2)$ in \eqref{t6a}, one can check that the supp$(\Phi)\cap \overline{(T_0-\delta,T_0)\times\Gamma_{out}}=\emptyset$. This means that the results of Theorem \ref{tt2}  and Corollary \ref{cc1} can be applied to the simultaneous determination of coefficients and internal source from a single measurement  separated, by an interval of time, from the application  of the single Dirichlet input $\Phi$. So far, only some small number of articles in the mathematical literature have been devoted to the recovery of coefficients or a manifold from  excitation and  measurements made on disjoint sets for general Riemannian manifolds or a bounded domain (see e.g. \cite{KKLO,KOSY,KLLY,LO}). All these results considered data on disjoint sets with respect to the space variables. As far as we know, the application of these class of results require a geometrical condition imposed  to the support of the inputs $\Gamma_{in}$ (like  the geometrical control condition of \cite{Bardos1992} considered in \cite{KKLO, LO}) and, as far as we know, these results allow only some local recovery of coefficients. In contrast to these results and their applications to fractional diffusion equations stated in \cite{KOSY,KLLY}, in Theorem \ref{tt2} (resp. in Corollary \ref{cc1}) we obtain, for what seems to be the first time, the full recovery of coefficients and manifolds from excitation and single measurement made on disjoint set of the lateral boundary $(0,T)\times\partial\tilde{\Omega}$ (resp. $(0,T)\times\partial M$) with respect to the time variable. Contrary to the works of \cite{KKLO,KOSY,KLLY,LO}, in Corollary \ref{cc1} we do not impose any geometrical condition to $\Gamma_{in}$ and in Theorem \ref{tt2}, Corollary \ref{cc1} we obtain the full recovery of coefficients.  In order to obtain these extensions of the works of  \cite{KLLY,KOSY}, we use an argument borrowed from \cite{KJ} which can only be applied in the case $\alpha\neq1$. For $\alpha=1$, it is not clear that condition \eqref{t6a} implies \eqref{t2aa}. This property emphasis the memory effect of time fractional diffusion equations.

One of the key ingredient in our proofs is based on a step by step argumentation allowing to transfer our inverse problems into a family of inverse problems that we solve separately. However, in Theorem \ref{t1} and in Corollary \ref{c1}, \ref{c2}, we need to consider the recovery of the obstacle, the order of derivation $\alpha$ in the situation where the coefficients of the equation are unknown. To overcome this difficulty, for the recovery of the obstacle we use the extra conditions \eqref{ob1}-\eqref{ob2} and \eqref{ob3}-\eqref{ob4}. However, we prove the recovery of the obstacle without the knowledge of $\alpha$. Moreover, we give a proof of the recovery of the order of derivation $\alpha$ in the context of Theorem \ref{t1} and in Corollary \ref{c1}, \ref{c2} without requiring the knowledge of the manifold and the different coefficient appearing in the equation (see Step 3 in the proof of Theorem \ref{t1} for more details). 

Let us observe that the recovery of the order of derivation $\alpha$ and the simultaneous recovery of the internal sources $\{u_0,f\}$ under assumption (vi) and (vi'), stated in Theorem \ref{t1} and in Corollary \ref{c1}, \ref{c2}, are new in their own. Indeed, in the Step 3 of the proof of Theorem \ref{t1}, we prove for what seems to be the first time, the recovery of the order of derivation $\alpha$ in an unknown medium (coefficients and manifold unknown) from a single boundary measurement associated with a single boundary input while other results seems to consider measurement at one internal point and internal excitation given by the initial condition (see e.g. \cite{HNWY,Ya1}). In addition, in the Step 5 of the proof of Theorem \ref{t1} we show how one can prove the simultaneous recovery of the initial condition $u_0$ and the source term $f$ under the assumption (vi) and (vi'). As far as we known, this is the first result stating the simultaneous recovery of the two internal sources $\{u_0,f\}$. Indeed, it seems that, all other comparable results (see e.g. \cite{JLLY,KSXY,LRY}), considered only the recovery the initial condition $u_0$ or the source term $f$ but not the simultaneous recovery of these two parameters.

\subsection{Outline}
This paper is organized as follows. In Section 3 we recall some properties of solutions of \eqref{eq1} when $T=\infty$, including some properties of analyticity in time of solutions \eqref{eq2}-\eqref{eq3} (see Section 3). Applying these results, in Section 4, 5 and 6 we complete the proof of our uniqueness results. Namely, in Section 4 we prove Theorem \ref{t1} while Section 5 (resp. 6) will be devoted to the proof of Theorem \ref{tt2} (resp. \ref{t2}). Finally, in Section 7 we give the proof of the Corollary \ref{c1}, \ref{cc1} and \ref{c2}.
\section{Analytic extension of solutions}
In this section we consider $(a,\rho,q)$ satisfying \eqref{ell}-\eqref{eq-rho} and $B\in L^\infty(\Omega)^d$. Let $k\in\mathbb N:=\{1,2,\ldots\}$, $\R_+=(0,+\infty)$ and consider the initial boundary value problems

\bel{eq2}
\left\{ \begin{array}{lll} 
&(\rho(x)\partial_t^{\alpha}v_0 -\textrm{div}\left(a(x) \nabla_x v_0 \right)+B\cdot\nabla_x v_0+q(x) v_0)(t,x)  =  F(t,x)\mathds{1}_{(0,T)}(t), & (t,x)\in  \R_+ \times \Omega,\\
&v_0(t,x)  =  0, & (t,x) \in  \R_+ \times \partial\tilde{\Omega}, \\
& v_0(t,x)  =  0, & (t,x) \in  \R_+ \times \partial\omega, \\
&\pd_t^\ell v_0(0,\cdot)  =  u_\ell, & \mbox{in}\ \Omega,\ \ell=0,...,\lceil\alpha\rceil-1,
\end{array}
\right.
\ee
\bel{eq3}
\left\{ \begin{array}{lll} 
&(\rho(x)\partial_t^{\alpha}v_k -\textrm{div}\left(a(x) \nabla_x v_k \right)+B\cdot\nabla_x v_k+q(x) v_k)(t,x)  =  0, & (t,x)\in  \R_+ \times \Omega,\\
&v_k(t,x)  =  d_k\psi_k(t)\chi(x)\eta_k(x), & (t,x) \in  \R_+ \times \partial\tilde{\Omega}, \\
& v_k(t,x)  =  0, & (t,x) \in  \R_+ \times \partial\omega, \\  
&\pd_t^\ell v_k(0,\cdot)  =  0, & \mbox{in}\ \Omega,\ \ell=0,...,\lceil\alpha\rceil-1.
\end{array}
\right.
\ee
Here $\mathds{1}_{(0,T)}$ denotes the characteristic function of $(0,T)$ and we refer to the beginning of Section 1.4 for the definition of the parameters $d_k$, $\psi_k$, $\chi$ and $\eta_k$. For simplicity, we will assume here that for $\alpha\in(1,2)$ we have $u_1\equiv0$. In the present paper, following \cite{KY1,KY3,SY}, we define the weak solutions of the problem
\bel{eq0}\left\{ \begin{array}{lll} 
&(\rho(x)\partial_t^{\alpha}v -\textrm{div}\left(a(x) \nabla_x v \right)+B\cdot\nabla_x v+q(x) v)(t,x)  =   F(t,x)\mathds{1}_{(0,T)}(t), & (t,x)\in  \R_+ \times \Omega,\\
&v(t,x)  =  h(t,x), & (t,x) \in  \R_+ \times \partial\tilde{\Omega}, \\
& v(t,x)  =  0, & (t,x) \in  \R_+ \times \partial\omega, \\  
&\pd_t^\ell v(0,\cdot)  =  u_\ell, & \mbox{in}\ \Omega,\ \ell=0,...,\lceil\alpha\rceil-1,
\end{array}
\right.\ee
 in the following way.
\begin{defn}\label{d1} Let  $F\in L^1(0,T;L^2(\Omega))$, $u_0\in L^2(\Omega)$ and $h\in L^1_{loc}(\R_+;H^{\frac{3}{2}}(\partial\tilde{\Omega}))$ satisfying
$$\inf\{\epsilon>0:\ e^{-\epsilon t}h\in L^1(\R^+; H^{\frac{3}{2}}(\partial\tilde{\Omega}))\}=0$$
We say that the problem \eqref{eq0}  admits a  weak solution $v$ if  $v\in L^1_{\textrm{loc}}(\R^+;L^2(\Omega))$ satisfies the following conditions:\\
1)  $p_*:=\inf\{\epsilon>0:\ e^{-\epsilon t}v\in L^1(\R^+; L^2(\Omega))\}<\infty$ and we can find $p_0\geq p_*$ independent of $F$, $u_0$ and $h$,\\
2) for all $p>p_0$ the Laplace transform in time $V(p)=\int_0^{+\infty}e^{-pt}v(t,.)dt$ of $v$ solves
\[\left\{\begin{aligned}\mathcal AV(p)+\rho(x)p^\alpha V(p)&=\int_0^Te^{-pt}F(t,\cdot)dt+ p^{\alpha-1}\rho u_{0},\quad &\textrm{in }\Omega,\\   V(p)&=0,\quad &\textrm{on }\partial\omega,\\ V(p)&=\int_0^{+\infty}e^{-pt}h(t,\cdot)dt,\quad &\textrm{on }\partial\tilde{\Omega},\end{aligned}\right.\]
where $$\mathcal Au=-\textrm{div}\left(a(x) \nabla_x u \right)+B(x)\cdot\nabla_xu +q(x)u,\quad u\in H^1(\Omega).$$ \end{defn}

One can easily check that the weak solution of \eqref{eq2}-\eqref{eq3} considered by \cite{KY1,KY2,SY} coincides with the one given by Definition \ref{d1}. Moreover, following \cite{KY1,KY2,SY}, we can deduce that, for all $k\in\mathbb N\cup\{0\}$, the problems \eqref{eq2}-\eqref{eq3} admit a unique weak solution $v_0\in\mathcal C((0,+\infty); H^{2\gamma}(\Omega))$, $v_k\in\mathcal C^1([0,+\infty); H^{2\gamma}(\Omega))$, $\gamma\in[0,1)$, $k\in\mathbb N$. Based on the above definition of weak solutions, we will recall some properties of analiticity in time of the solution of  problems \eqref{eq2}-\eqref{eq3}. More precisely, for $k\in\mathbb N$, we fix  $\epsilon_k\in (0,(t_{2k}-t_{2k-1})/2)$ and we set $$  D_{s,\theta}=\{s+re^{i\beta}:\ \beta\in(-\theta,\theta),\ r>0\},\quad s,\theta\in[0,+\infty).$$ Here $(t_k)_{k\in\mathbb N}$ denotes the sequence introduced at the beginning of Section 1.4. For any open set $ U$ of $\mathbb C$ or of $\mathbb R$, and $X$ a Banach space, we denote by $\mathcal H(U;X)$ the set of analytic functions on $ U$ taking values in $X$.
For $B\equiv0$, combining \cite[Proposition 3.1]{KLLY} with \cite[Proposition 2.1]{KSXY}, we obtain the following analytic extension result.
\begin{thm}
\label{t4} Assume that  $B\equiv0$. Let $\epsilon_0\in(0,\tau_1/3)$ be such that 
\bel{t4a} \textrm{supp}(F)\subset [0,\tau_1-3\epsilon_0]\times\overline{\Omega}\ee
and let $\theta\in\left(0,\min\left(\frac{\frac{\pi}{\alpha}-\frac{\pi}{2}}{2},\frac{\pi}{4}\right)\right)$.
Then, for all $\gamma\in[0,1)$, the solution $v_0$ of \eqref{eq2}  can be extended uniquely to a function $\tilde{v}_0\in L^1(0,\tau_1-\epsilon_0;H^{2\gamma}(\Omega))\cap \mathcal C(\overline{\mathcal  D_{\tau_1-\epsilon_0,\theta}};H^{2\gamma}(\Omega))\cap \mathcal H(\mathcal  D_{\tau_1-\epsilon_0,\theta};H^{2\gamma}(\Omega))$. Moreover, for any $k\in\mathbb N$,  the solution $v_k$ of \eqref{eq2}  can be extended uniquely to a function $\tilde{v}_k\in \mathcal C^1([0,t_{2k-1}+\epsilon_k]\cup\overline{\mathcal  D_{t_{2k-1}+\epsilon_k,\theta}};H^2(\Omega))\cap \mathcal H(\mathcal  D_{t_{2k-1}+\epsilon_k,\theta};H^2(\Omega))$.
\end{thm}

Now let us consider the case $B\in\mathcal C(\overline{\Omega})^d$ a non-uniformly vanishing function, $\rho\in \mathcal C(\overline{\Omega})$, $a\equiv1$ and $q\equiv0$. We consider the following result.
\begin{thm}
\label{t5} Assume that the condition  \eqref{t4a}  is fulfilled,  $a\equiv1$, $q\equiv0$ and $\alpha\in(0,1]$. 
Then, there exists $\theta\in\left(0,\min\left(\frac{\frac{\pi}{\alpha}-\frac{\pi}{2}}{2},\frac{\pi}{4}\right)\right)$ such that, for any $\gamma\in(0,1)$, the solution $v_0$ of \eqref{eq2}  can be extended uniquely to a function $\tilde{v}_0\in  L^1(0,\tau_1-\epsilon_0;H^{2\gamma}(\Omega))\cap \mathcal C(\overline{\mathcal  D_{\tau_1-\epsilon_0,\theta}};H^{2\gamma}(\Omega))\cap \mathcal H(\mathcal  D_{\tau_1-\epsilon_0,\theta};H^{2\gamma}(\Omega))$. Moreover, for any $k\in\mathbb N$,  the solution $v_k$ of \eqref{eq2}  can be extended uniquely to a function $\tilde{v}_k\in \mathcal C^1([0,t_{2k-1}+\epsilon_k]\cup\overline{\mathcal  D_{t_{2k-1}+\epsilon_k,\theta}};H^{2\gamma}(\Omega))\cap \mathcal H(\mathcal  D_{t_{2k-1}+\epsilon_k,\theta};H^{2\gamma}(\Omega))$.
\end{thm}

The second claim of this theorem can be deduced from \cite[Proposition 3.2]{KLLY}. However, we will give here a simplified proof not based on iteration arguments. For this purpose, we fix  $A$ the unbounded elliptic operator defined by $A=\rho^{-1}\mathcal A$ acting on $L^2(\Omega;\rho dx)$ with domain $D(A)=H^1_0(\Omega)\cap H^2(\Omega)$. According to \cite[Theorem 2.1]{A} (see also \cite[Theorem 2.5.1]{LLMP}), there exists $\theta_0\in\left(\frac{\pi}{2},\pi\right)$ and $r_0\geq0$ such that the set $ D_{r_0,\theta_0}$ is in the resolvent set of $A$. Moreover, there exists $C>0$, depending on $\mathcal A$, $\rho$,  $\Omega$, such that
\bel{ess}\norm{(A+z)^{-1}}_{\mathcal B(L^2(\Omega;\rho dx))}+|z-r_0|^{-1}\norm{(A+z)^{-1}}_{\mathcal B(L^2(\Omega;\rho dx);H^2(\Omega))}\leq C|z-r_0|^{-1},\quad z\in  D_{r_0,\theta_0}.\ee
Here we use the fact that, thanks to \eqref{eq-rho}, $L^2(\Omega)=L^2(\Omega;\rho dx)$ with equivalent norms.
We fix $\theta_1\in\left(\frac{\pi}{2},\frac{\pi}{2}+\frac{\theta_0-\frac{\pi}{2}}{2}\right)$, $\delta \in (0,+\infty)$ and we consider $ \gamma(\delta,\theta_1)$ the contour in $\C$ defined by
$$ 
\gamma(\delta,\theta_1) := \gamma_-(\delta,\theta_1) \cup \gamma_0(\delta,\theta_1) \cup \gamma_+(\delta,\theta_1)$$
oriented in the counterclockwise direction, where
$$
\gamma_0(\delta,\theta_1) := \{ r_1+\delta e^{i \beta};\ \beta \in
[-\theta_1,\theta_1] \}\ \mbox{and}\ \gamma_\pm(\delta,\theta_1) 
:= \{r_1+s e^{\pm i \theta_1};\ s \in [\delta,+\infty) \}$$
and the two copies of the $\pm$ sign in the above identity must both be replaced in the same way. Here we choose $r_1>r_0$ large enough and in particular,  for all $\delta>0$, we have $\gamma(\delta,\theta_1)\subset D_{r_0,\theta_0}$.
Let  $\theta_2\in \left(0,\theta_1-\frac{\pi}{2}\right)$. Applying  the above properties of the operator $A$,  for $\alpha\in(0,1]$ and $z\in D_{0,\theta_2}$, we can define the operator $S(z)\in\mathcal B(L^2(\Omega))$ by 
$$S(z)u_0=\frac{1}{2i\pi}\int_{\gamma(\delta,\theta_1)}e^{zp}(A+ p^\alpha)^{-1}u_0dp,\quad u_0\in L^2(\Omega).$$
We  consider first the following property of the map $z\mapsto S(z)$.

\begin{lem}\label{l2} For all $\gamma\in[0,1]$, the map $z\mapsto S(z)$ is lying in $\mathcal H(D_{0,\theta_2};\mathcal B(L^2(\Omega); H^{2\gamma}(\Omega)))$ and there exists $C>0$ depending only on $\mathcal A$, $\rho$ and $\Omega$ such that
\bel{l2a} \norm{S(z)}_{B(L^2(\Omega;\rho dx); H^{2\gamma}(\Omega))}\leq C\max(|z|^{\alpha(1-\gamma)-1},1)e^{r_1\re(z)},\quad z\in D_{0,\theta_2}.\ee

\end{lem}
\begin{proof} 
In all this proof $C$ is a constant depending only on $\mathcal A$, $\rho$ and $\Omega$ that may change from line to line.
Using the fact that by interpolation \eqref{ess} implies that
\bel{l2b}\norm{(A+ p^\alpha)^{-1}}_{\mathcal B(L^2(\Omega;\rho dx); H^{2\gamma}(\Omega))}\leq C\abs{|p|^{\alpha}-r_0}^{-(1-\gamma)},\quad p\in\gamma(\delta,\theta_1),\ee
one can easily check that 
$$S\in\mathcal H(D_{0,\theta_2};\mathcal B(L^2(\Omega); H^{2\gamma}(\Omega))).$$
Now let us show the estimate \eqref{l2a}. Fix $z\in D_{0,\theta_2}$. Using the fact that $p\mapsto (A+p^\alpha)^{-1}$ is analytic on $p\in D_{r_0,\theta_0}$ and applying \eqref{l2b} combined with some arguments used  in \cite[Lemma 2.4]{KSY}, one can check  that $S(z)=S_-(z)+S_0(z)+S_+(z)$ with 
$$S_m(z)=\frac{1}{2i\pi}\int_{\gamma_m(|z|^{-1},\theta_1)}e^{zp}(A+ p^\alpha)^{-1}dp,\quad m=0,\mp,\ z\in D_{0,\theta_2}.$$
Therefore, the lemma will be completed if we prove that
\bel{l2c}\norm{S_m(z)}_{B(L^2(\Omega;\rho dx); H^{2\gamma}(\Omega))}\leq C\max(|z|^{\alpha(1-\gamma)-1},1)e^{r_1\re(z)},\quad z\in D_{0,\theta_2},\  m=0,\mp.\ee
For $m=0$, applying \eqref{l2b}, we find
$$\begin{aligned}\|S_0(z)\|_{\mathcal B(L^2(\Omega;\rho dx); H^{2\gamma}(\Omega))}\leq &C\int_{-\theta_1}^{\theta_1}e^{r_1\re(z)}|z|^{-1}\norm{\left(A+ (r_1+|z|^{-1}e^{i\beta})^\alpha\right)^{-1}}_{B(L^2(\Omega;\rho dx); H^{2\gamma}(\Omega))}d\beta\\
&\leq C \abs{(r_1+|z|^{-1}e^{i\beta})^\alpha-r_0}^{-(1-\gamma)}|z|^{-1}e^{r_1\re(z)}\\
&\leq C\max\left(|z|^{\alpha(1-\gamma)-1},1\right)e^{r_1\re(z)},\quad z\in D_{0,\theta_2},\end{aligned}$$
which clearly implies \eqref{l2c} for $m=0$. Here in the last inequality we have used the fact that $r_1>r_0$ is chosen sufficiently large. Now let us consider the case $m=\mp$. For any $z\in D_{0,\theta_2}$, we find
$$\|S_\mp(z)\|_{\mathcal B(L^2(\Omega;\rho dx); H^{2\gamma}(\Omega))}\leq Ce^{r_1\re(z)}\int_{|z|^{-1}}^{+\infty}e^{r |z|\cos(\theta_1+\textrm{arg}(z)) }\|(A+ (re^{i\theta})^\alpha)^{-1}\|_{B(L^2(\Omega;\rho dx);H^{2\gamma}(\Omega))}dr,$$
with $C>0$ independent of $z$.
Applying again \eqref{l2b}, for any $z\in D_{0,\theta_2}$, we obtain
$$\begin{aligned}\|S_\mp(z)\|_{\mathcal B(L^2(\Omega;\rho dx); H^{2\gamma}(\Omega))}&\leq Ce^{r_1\re(z)}\int_{|z|^{-1}}^{+\infty}e^{r |z|\cos(\theta_1+\textrm{arg}(z)) }\abs{(r_1+re^{i\beta})^\alpha-r_0}^{-(1-\gamma)}dr\\
\ &\leq Ce^{r_1\re(z)}\int_0^{+\infty}e^{r|z| \cos(\theta_1-\theta_2) }\abs{(r_1+re^{i\beta})^\alpha-r_0}^{-(1-\gamma)}dr\\
\ &\leq C\max(|z|^{\alpha(1-\gamma)-1},1)e^{r_1\re(z)}\int_0^{+\infty}e^{t \cos(\theta_1-\theta_2)}\max(t^{-(1-\gamma)\alpha},1)dt.\end{aligned}$$
Therefore, using the fact that $\theta_1-\theta_2\in\left(\frac{\pi}{2},\pi\right)$, we deduce that \eqref{t1a} holds also true for $m=\mp$. This completes the proof of the lemma.\end{proof}

In addition to these properties, by combining estimate \eqref{l2a} with the arguments of \cite[Theorem 1.1]{KSY} and \cite[Remark 1]{KSY}, we deduce that , for $F\in L^\infty(0,T;L^2(\Omega))$ satisfying \eqref{t4a}, with $u_0=u_{\lceil\alpha\rceil-1}=0$, \eqref{eq2} admits a unique weak solution $v_0\in L^1_{loc}([0,+\infty);H^{2\gamma}(\Omega))$, $\gamma\in(0,1)$, taking the form
\bel{t2cd}v_0(t,\cdot)=\int_0^tS(t-s)F(s,\cdot)\mathds{1}_{(0,T)}(s)ds,\quad t\in\R_+.\ee
Using some arguments similar to \cite[Proposition 6.1.]{KSXY}, one can show that the identity \eqref{t2cd} holds true for source terms $F$ lying in $L^1(0,T;L^2(\Omega))$. Armed with this result we are now in position to complete the proof of Theorem \ref{t5}.

\textbf{Proof of Theorem \ref{t5}.} We start with the first claim of Theorem \ref{t5}. For $F=0$, the analytic extension of $v_0$ can be deduced easily from arguments similar to the proof of \cite[Theorem 2.3]{KSY}. For this purpose, without loss of generality we assume that $u_0\equiv0$. Then we fix 
$$\tilde{v}_0(z,\cdot)= \int_0^{\tau_1-3\epsilon_0}S(z-s)F(s,\cdot)ds,\quad z\in D_{\tau_1-2\epsilon_0,\theta_2}.$$
Applying Lemma \ref{l2}, we deduce that  $\tilde{v}_0\in\mathcal H(D_{\tau_1-2\epsilon_0,\theta_2};H^{2\gamma}(\Omega))$ and applying \eqref{t4a}, we obtain
$$\tilde{v}_0(t,\cdot)=\int_0^{\tau_1-3\epsilon_0}S(t-s)F(s,\cdot)ds=\int_0^{t}S(t-s)F(s,\cdot)ds=v_0(t,\cdot),\quad t\in(\tau-\epsilon_0,+\infty).$$
This clearly implies the first claim of the theorem. 

Now let us consider the second claim of the theorem. For this purpose, we fix $k\in\mathbb N$ and we consider $\delta_k\in(0,\epsilon_k/3)$. Let us consider $u_k$ solving
$$\left\{ \begin{array}{lll} 
&(\rho(x)\partial_t^{\alpha}u_k -\Delta u_k)(t,x)   =  0, & (t,x)\in  \R_+ \times \Omega,\\
&u_k(t,x)  =  d_k\psi_k(t)\chi(x)\eta_k(x), & (t,x) \in  \R_+ \times \partial\tilde{\Omega}, \\
&u_k(t,x)=0,& (t,x) \in  \R_+ \times \partial\omega, \\  
&\pd_t^\ell u_k(0,\cdot)  =  0, & \mbox{in}\ \Omega,\ \ell=0,...,\lceil\alpha\rceil-1.
\end{array}
\right.$$
Using the above properties, we deduce that the solution $v_k$ of \eqref{eq3} is given by
\bel{t55}v_k(t,\cdot)=u_k(t,\cdot)-\int_0^{t}S(t-s)B\cdot\nabla_x u_k(s,\cdot)ds,\quad t\in\R_+.\ee
Here according to estimate \eqref{l2a}, we can chose $p_0=r_1$. In view of Theorem \ref{t4},  $u_k$ can be extended to $\tilde{u}_k\in \mathcal C^1([0,t_{2k-1}+\epsilon_k]\cup\overline{\mathcal  D_{t_{2k-1}+\delta_k,\theta_2}};H^2(\Omega))\cap\mathcal H( D_{t_{2k-1}+\delta_k,\theta_2};H^{2}(\Omega))$. Therefore, we can define
\bel{t5c}\tilde{v}_k(z,\cdot)=\tilde{u}_k(z,\cdot)+\tilde{w}_k(z,\cdot)+ \tilde{y}_k(z,\cdot),\quad z\in D_{t_{2k-1}+\delta_k,\theta_2},\ee
with 
$$\tilde{w}_k(z,\cdot)=-\int_0^{t_{2k-1}+2\delta_k}S(z-s)B\cdot\nabla_x u_k(s,\cdot)ds,$$
$$\tilde{y}_k(z,\cdot)=-\int_{0}^{z-t_{2k-1}-2\delta_k}S(p)B\cdot\nabla_x \tilde{u}_k(z-p,\cdot)dp.$$
It is clear that 
$$\begin{aligned}\tilde{v}_k(t,\cdot)&=u_k(t,\cdot)-\int_0^{t_{2k-1}+2\delta_k}S(t-s)B\cdot\nabla_x u_k(s,\cdot)ds-\int_{0}^{t-t_{2k-1}-2\delta_k}S(s)B\cdot\nabla_x \tilde{u}_k(t-s,\cdot)ds\\
\ &=u_k(t,\cdot)-\int_0^{t}S(t-s)B\cdot\nabla_x u_k(s,\cdot)ds,\quad t\in(t_k+\epsilon_k,+\infty).\end{aligned}$$
Combining this with \eqref{t55}, one can check that $\tilde{v}_k$ extends $v_k$. Therefore, using the fact that $\overline{D_{t_{2k-1}+\epsilon_k,\theta_2}}\subset D_{t_{2k-1}+2\delta_k,\theta_2}$, the proof will be completed if we prove that $\tilde{v}_k\in \mathcal H( D_{t_{2k-1}+2\delta_k,\theta_2};H^{2\gamma}(\Omega))$. For this purpose, we only need to show that $\tilde{w}_k$ and $\tilde{y}_k$ are lying in $\mathcal H( D_{t_{2k-1}+2\delta_k,\theta_2};H^{2\gamma}(\Omega))$. For $\tilde{w}_k$, we first fix $\delta_*\in (0,t_{2k-1}+2\delta_k)$ and we consider
$$\tilde{w}_{k,\delta_*}:=-\int_0^{t_{2k-1}+2\delta_k-\delta_*}S(z-s)B\cdot\nabla_x u_k(s,\cdot)ds.$$
Repeating the arguments used at the beginning of this proof, we deduce that $\tilde{w}_{k,\delta_*}\in \mathcal H( D_{t_{2k-1}+2\delta_k,\theta_2};H^{2\gamma}(\Omega))$. Moreover, for any compact set $K\subset D_{t_{2k-1}+2\delta_k,\theta_2}$, applying \eqref{l2a}, for all $z\in K$, we get
$$\norm{\tilde{w}_{k,\delta_*}(z,\cdot)-\tilde{w}_k(z,\cdot)}_{H^{2\gamma}(\Omega)}\leq C\norm{u_k}_{L^\infty(0,T;H^1(\Omega))}\int_{t_{2k-1}+2\delta_k-\delta_*}^{t_{2k-1}+2\delta_k}\max(|z-s|^{(1-\gamma)\alpha-1},1)ds.$$
This proves that $\tilde{w}_{k,\delta_*}$ converges uniformly,  with respect to $z\in K$, as $\delta_*\to0$ to $\tilde{w}_k$. Therefore, we have $\tilde{w}_k\in \mathcal H( D_{t_{2k-1}+2\delta_k,\theta_2};H^{2\gamma}(\Omega))$. For $\tilde{y}_k$, combining the fact that $\tilde{u}_k\in \mathcal H( D_{t_{2k-1}+\delta_k,\theta_2};H^{2}(\Omega))$ with the estimate \eqref{l2a} and the fact that, according to Lemma \ref{l2}, $S\in \mathcal H(D_{0,\theta_2};\mathcal B(L^2(\Omega;\rho dx); H^{2\gamma}(\Omega)))$, one can check (see also step 2 in the proof of \cite[Proposition 3.2]{KLLY}) that $\tilde{y}_k\in \mathcal H(D_{t_{2k-1}+2\delta_k,\theta_2}; H^{2\gamma}(\Omega))$. This completes the proof of the theorem.\qed

Using the analiticity properties described above, we will complete the proof of our main results in the coming sections.
We start with Theorem \ref{t1},  \ref{tt2}, \ref{t2}. Then, we prove   Corollary \ref{c1}, \ref{cc1}, \ref{c2} in the last section. We will only give the detail of the proof of Theorem \ref{t1}, \ref{tt2} and \ref{t2}. For Corollary \ref{c1}, \ref{cc1}, \ref{c2}, we will mainly adapt to the framework of manifolds the arguments used in Theorem \ref{t1}, \ref{tt2}.

\section{Proof of Theorem \ref{t1}.}  

The proof of Theorem \ref{t1} will be decomposed into five steps. We start by proving that \eqref{t1d} implies that, for all $k\in\mathbb N\cup\{0\}$, we have
\bel{t1e}a_1(x)\partial_\nu v_k^{1}(t,x)=a_2(x)\partial_\nu v_k^{2}(t,x),\quad (t,x)\in(0,+\infty)\times \Gamma_{out},\ee
with $v_0^j$ the solution of \eqref{eq2} for $B\equiv0$, $\alpha=\alpha_j$, $\omega=\omega_j$ and $(a,\rho,q, u_0,u_{1},F)=(a_j,\rho_j,q_j,u_0^j,0,F_j)$, $j=1,2$, and $v_k^{j}$, $j=1,2$, $k\in\mathbb N$, the solution of \eqref{eq3} for $B\equiv0$, $\omega=\omega_j$ and $(a,\rho,q)=(a_j,\rho_j,q_j)$, $j=1,2$.
Using \eqref{t1e} with $k=1$ and exploiting condition \eqref{ob1}-\eqref{ob2}, we will deduce that $\omega_1=\omega_2$. Then, applying \eqref{t1e}, with $k=1$, we get $\alpha_1=\alpha_2$. After that, using \eqref{t1e}, with  $k\in\mathbb N$, we will obtain
\bel{t1ee}a_1=a_2,\quad \rho_1=\rho_2,\quad q_1=q_2.\ee
Finally, combining all these results and applying \eqref{t1e} with $k=0$ we will get
\bel{t1eee}u_0^1=u_0^2,\quad f_1=f_2.\ee

\textbf{Step 1.} We will prove \eqref{t1e} by iteration. Let us start with $k=0$. For this purpose, using the properties of the sequence $(\psi_k)_{k\geq 1}$, let us observe that
$$\psi_k(t)=0,\quad k\geq1,\ t\in(0,t_0)=(0,\tau_1).$$
Therefore, the restriction of $u^{j}$  to $(0,\tau_1)\times\Omega_j$ solves the boundary value problem
\begin{equation}\label{beq}\begin{cases} 
(\rho_j(x)\partial_t^{\alpha}u^j -\textrm{div}(\nabla_x a_ju^j)(t,x) +q_ju^j(t,x)  =  \sigma(t) f_j(x), & (t,x)\in  (0,\tau_1) \times \Omega_j,\\
u^j(t,x)  =  0, & (t,x) \in  (0,\tau_1)\times \partial\tilde{\Omega}, \\
u^j(t,x)=0,& (t,x) \in  (0,\tau_1) \times \partial\omega_j, \\  
\begin{cases}
u^j=u_0^j & \mbox{if }0<\alpha\leq1,\\
u^j=u_0,\quad \partial_t u^j=0 & \mbox{if }1<\alpha<2,
\end{cases} & \mbox{in }\{0\}\times \Omega_j.
\end{cases}\end{equation}
Using the fact that the restriction of $v_0^{j}$  to $(0,\tau_1)\times\Omega_j$ solves also \eqref{beq} and applying the uniqueness of the solution of \eqref{beq}, we get
$$v_0^{j}(t,x)=u^{j}(t,x),\quad (t,x)\in(0,\tau_1)\times\Omega_j$$
and condition \eqref{t1d} implies
\bel{t1f}a_1(x)\partial_\nu v_0^{1}(t,x)=a_2(x)\partial_\nu v_0^{2}(t,x),\quad (t,x)\in(0,\tau_1)\times \Gamma_{out}.\ee
Combining this with \eqref{source}-\eqref{source2} and applying Theorem \ref{t4}, we deduce that there exits $\epsilon_1\in(0,\tau_1)$ such that $v_0^j\in \mathcal A((\tau_1-\epsilon_1,+\infty);H^{\frac{7}{4}}(\Omega))$, $j=1,2$, which implies that $\partial_\nu v_0^j\in \mathcal A((\tau_1-\epsilon_1,+\infty);L^2(\partial\tilde{\Omega}))$. Therefore, \eqref{t1f} implies \eqref{t1e} for $k=0$. 

Now let us consider $\ell\geq0$ and 	assume that \eqref{t1e} is fulfilled for $k=0,\ldots,\ell$.  Since
$$\psi_m(t)=0,\quad m\geq\ell +2,\ t\in(0,t_{2\ell+2}),$$
we know that 
$$\sum_{k=0}^{\ell+1}v_k^{j}(t,x)=u^{j}(t,x),\quad (t,x)\in(0,t_{2\ell+2})\times\Omega.$$
Therefore, \eqref{t1d} implies
$$\sum_{k=0}^{\ell+1}a_1(x)\partial_\nu v_k^{1}(t,x)=\sum_{k=1}^{\ell+1}a_2(x)\partial_\nu v_k^{2}(t,x),\quad (t,x)\in(0,t_{2\ell+2})\times \Gamma_{out}.$$
Then, from our iteration assumption we deduce that
$$a_1(x)\partial_\nu v_{\ell+1}^{1}(t,x)=a_2(x)\partial_\nu v_{\ell+1}^{2}(t,x),\quad (t,x)\in(0,t_{2\ell+2})\times \Gamma_{out}.$$
Therefore, applying again Theorem \ref{t4} we deduce that $t\mapsto \partial_\nu v_{\ell+1}^{j}(t,\cdot)_{|\Gamma_{out}}\in \mathcal A((t_{2\ell+1}+\epsilon_\ell,+\infty);L^2(\Gamma_{out}))$, $j=1,2$, and we get \eqref{t1e} for $k=\ell+1$. This proves that \eqref{t1e} holds true for all $k\in\mathbb N\cup\{0\}$.\\

\textbf{Step 2.} We will now show that \eqref{t1e} with $k=1$ implies that $\omega_1=\omega_2$. Let us fix $V_1^j(p,x)$ the Laplace transform in time, at $p>0$, of the solution $v_1^j$ of the  problem \eqref{eq2} for $k=1$. The definition of weak solution of \eqref{eq1} implies that, for all $p>0$, $V_1^j(p,\cdot)$ solves
\bel{t11a}\left\{\begin{aligned}-\textrm{div}\left(a_j \nabla_x V_1^j(p,\cdot) \right)+q_jV_1^j(p,\cdot) +\rho_jp^{\alpha_j} V_1^j(p,\cdot)&=0,\quad &\textrm{in }\Omega_j,\\   V_1^j(p,\cdot)&=\hat{\psi}_1(p)\chi\eta_1,\quad &\textrm{on }\partial\tilde{\Omega},\\ V_1^j(p,\cdot)&=0,\quad &\textrm{on }\partial\omega_j,\end{aligned}\right.\ee
where 
$$\hat{\psi}_1(p)=\int_0^{+\infty}e^{-pt}\psi_1(t)dt,\quad p>0.$$
Following the arguments used in Step  2 of the proof of \cite[Theorem 2.2]{KLLY}, one can check that, for $s\in(3/2,2)$ and for all $p>0$, we have $t\mapsto e^{-pt}v_1^j\in L^1(\R_+;H^s(\Omega_j))$, $j=1,2$. Therefore, we can apply the Laplace transform in time to the identity \eqref{t1e}, with $k=1$, in order to get
$$a_1(x)\partial_\nu V_1^1(p,x)=a_2(x)\partial_\nu V_1^2(p,x),\quad p>0,\ x\in\Gamma_{out}.$$
Choosing $p=1$, we deduce from this  identity  that
\bel{t11b}a_1(x)\partial_\nu V_1^1(1,x)=a_2(x)\partial_\nu V_1^2(1,x),\quad x\in\Gamma_{out}.\ee
Combining \eqref{t11b}  with \eqref{ob1}-\eqref{ob2}, we deduce that the restriction of $V_1(1,\cdot)=V_1^1(1,\cdot)-V_1^2(1,\cdot)$ to $\mathcal O$ satisfies the conditions
\[\left\{\begin{aligned}-\textrm{div}\left(a_1 \nabla_x V_1(1,\cdot) \right)+q_1V_1(1,\cdot) +\rho_1 V_1(1,\cdot)&=0\quad &\textrm{in }\mathcal O,\\   V_1(1,\cdot)=\partial_\nu V_1(1,\cdot)&=0,\quad &\textrm{on }\Gamma_{out}\cap \partial\mathcal O.\end{aligned}\right.\]
Since $\mathcal O$ is connected, applying results of unique continuation for elliptic equations, we find
$$V_1(1,x)=0,\quad x\in\mathcal O.$$
Using the fact that $\partial(\omega_1\cup\omega_2)\subset\partial\mathcal O$, we deduce from this identity that
\bel{t111} V_1^1(1,x)= V_1(1,x)=0,\quad x\in(\partial\omega_2)\setminus(\partial\omega_1).\ee
Moreover, since $\chi\eta_1\in W^{2-\frac{1}{r},r}(\partial\Omega)$,  and $V_1^1(p,\cdot)$ solves \eqref{t11a}, \cite[Theorem 2.4.2.5]{Gr} implies that $V_1^1(1,\cdot)\in W^{2,r}(\Omega)$. Using the fact that $r>d/2$, the Sobolev embedding theorem implies that $V_1^1(1,\cdot)\in \mathcal C(\overline{\Omega})\cap H^1(\Omega)$. Fixing $\omega_*=\omega_2\setminus\overline{\omega_1}$, we deduce from \eqref{t111} that $V_1^1(1,\cdot)\in \mathcal C(\overline{\omega_*})\cap H^1(\omega_*)$ and $V_1^1(1,\cdot)=0$ on $\partial\omega_*$. Therefore, applying \cite[Theorem 9.17]{Br} and \cite[Remark 19]{Br}, we deduce that  the restriction of $V_1^1(1,\cdot)$ to $\omega_*$ is lying in $H^1_0(\omega_*)$. It follows that $V_1^1(1,\cdot)\in H^1_0(\omega_*)$ satisfies
\bel{t11c} -\textrm{div}\left(a_1 \nabla_x V_1^1(1,\cdot) \right)+q_1V_1^1(1,\cdot) +\rho_1 V_1^1(1,\cdot)=0\quad \textrm{in }\omega_*.\ee
Let us fix $H_*h=-\textrm{div}\left(a_1 \nabla_x h \right)+q_1h +\rho_1 h$ with domain $$D(H_*)=\{h\in H^1_0(\omega_*):\ -\textrm{div}\left(a_1 \nabla_x h \right)+q_1h +\rho_1 h\in L^2(\omega_*)\}.$$
Then, condition \eqref{t11c} implies that $V_1^1(1,\cdot)\in D(H_*)$ and $H_*V_1^1(1,\cdot)\equiv0$. On the other hand, since  $q_1\geq 0$ and $\rho_1>0$, one can check that $0$ is not in the spectrum of $H_*$  which implies that
\bel{t11d}V_1^1(1,x)=0,\quad  x\in\omega_*=\omega_2\setminus\overline{\omega_1}.\ee
Combining this with the fact that
$$-\textrm{div}\left(a_1 \nabla_x V_1^1(1,\cdot) \right)+q_1V_1^1(1,\cdot) +\rho_1 V_1^1(1,\cdot)=0\quad \textrm{in }\Omega_1$$
and applying results of unique continuation for elliptic equations we deduce that $V_1^1(1,\cdot)=0$ on $\Omega_1$. On the other hand, since $\psi_1\geq0$ and $\psi_1\not\equiv0$, one can check that $\hat{\psi}_1(1)\neq0$ and we obtain 
$$\chi\eta_1(x)=\frac{V_1^1(1,x)}{\hat{\psi}_1(1)}=0,\quad x\in\partial\tilde{\Omega}.$$
This contradicts the fact that $\chi\eta_1\not\equiv0$. Therefore, we have $\omega_1=\omega_2$.

\textbf{Step 3.}  In this step, we will show  that condition \eqref{t1e}, for $k=1$, implies that $\alpha_1=\alpha_2$. For this purpose, we will start by considering the asymptotic behavior of $\partial_\nu v_1^j(t,\cdot)|_{\Gamma_{out}}$, $j=1,2$, as $t\to+\infty$. Then, combining this asymptotic property with \eqref{t1e}, for $k=1$, we will deduce that $\alpha_1=\alpha_2$. We mention that a similar approach has been considered by other authors (see e.g. \cite{HNWY,Ya1}) with Dirichlet measurement at one internal point (a point $x_0\in\Omega$) in order to prove the recovery of the order of derivation $\alpha$. However, to the best of our knowledge this result will be the first one in that category stated in an unknown medium (since we have not yet proved that $(a_1,\rho_1,q_1)=(a_2,\rho_2,q_2)$), a Neumann boundary measurement and with a Dirichlet input (it seems that all other related results in that category have been stated with non-uniformly vanishing and known initial condition). From now on we fix $\Omega=\Omega_1=\Omega_2=\tilde{\Omega}\setminus\overline{\omega_1}$ which, according to our assumptions, is a $\mathcal C^2$ open connected set. Let us consider the operator $A_j$, $j=1,2$, acting on $L^2(\Omega_j;\rho_jdx)$ with domain $D(A_j)=H^2(\Omega)\cap H^1_0(\Omega)$ defined  by
\[
A_j w:=\frac{-\textrm{div}(a_j\nabla_x w)+q_j w}{\rho_j},\quad w\in D(A_j).
\]
Recall that here we associate with the weighted space $L^2(\Omega_j;\rho_jdx)$ the scalar product
$$\left\langle f,g\right\rangle_{L^2(\Omega_j;\rho_jdx)}=\int_{\Omega_j}f\overline{g}\rho_jdx.$$
Using the fact that the operator $\rho_jA_j$ is a selfadjoint operator acting on $L^2(\Omega_j)$, one can easily check that $A_j$ is a selfadjoint operator acting on $L^2(\Omega_j;\rho_jdx)$. Combining this with the fact that $A_j$ has a compact resolvent, we deduce that its spectrum consist of a strictly increasing sequence of eigenvalues. We fix $\{\lambda^j_k\}_{k\in\mathbb N}$ and $m^j_k\in\mathbb N$  the strictly increasing sequence of the eigenvalues of $A_j$ and the algebraic multiplicity of $\lambda^j_k$, respectively. For each eigenvalue $\lambda^j_k$, we introduce a family $\{\phi^j_{k,\ell}\}_{\ell=1}^{m^j_k}$ of eigenfunctions of $A_j$, i.e.,
\[
A_j\phi^j_{k,\ell}=\lambda^j_k\phi^j_{k,\ell},\quad\ell=1,\ldots,m^j_k,
\]
which forms an orthonormal basis in $L^2(\Omega;\rho_j dx)$ of the algebraic eigenspace of $A_j$ associated with $\lambda^j_k$. We introduce also, for $\beta_1,\beta_2>0$, the Mittag-Leffler function $E_{\beta_1,\beta_2}$ given by
$$
E_{\beta_1,\beta_2}(z)=\sum_{k=0}^\infty \frac{z^k}{\Gamma(\beta_1 k+\beta_2)},
\quad   z\in\mathbb C.
$$
We recall also that $E_{1,1}(z)=e^z$. Using this last property, we will give a unified representation of solutions of \eqref{eq2}-\eqref{eq3} including representation of solutions of classical parabolic equations (case $\alpha=1$). In view of \cite[Theorem 1.3]{KY3}, for $j=1,2$, one can check that, for all $t\in(0,+\infty)$, we have
$$v_1^j(t,\cdot)=\sum_{k=1}^\infty\sum_{\ell=1}^{m_k} \left(-\int_0^t(t-s)^{\alpha_j-1}E_{\alpha_j,\alpha_j}(-\lambda^j_k(t-s)^{\alpha_j})\psi_1(s)\left\langle \chi\eta_1,a_j\partial_{\nu}\phi_{k,\ell}^j\right\rangle_{L^2(\partial\Omega)}ds\right)\phi_{k,\ell}^j.$$
Using the fact  that $c_1=0$, we have supp$(\psi_1)\subset [0,\tau_2]$ and it follows that, for all $t\in(\tau_2,+\infty)$, we find
\bel{t111a}v_1^j(t,\cdot)=\sum_{k=1}^\infty\sum_{\ell=1}^{m_k} \left(-\int_0^{\tau_2}(t-s)^{\alpha_j-1}E_{\alpha_j,\alpha_j}(-\lambda^j_k(t-s)^{\alpha_j})\psi_1(s)\left\langle \chi\eta_1,a_j\partial_{\nu}\phi_{k,\ell}^j\right\rangle_{L^2(\partial\Omega)}ds\right)\phi_{k,\ell}^j.\ee
On the other hand, applying  \cite[Theorem 1.4, page 33-34]{P} one can check that there exists a constant $C>0$ such that
$$\abs{E_{\alpha_j,\alpha_j}(-\lambda t^{\alpha_j})}\leq Ct^{-2\alpha_j}\lambda^{-2},\quad \lambda\in(0,+\infty),\ t\in(0,+\infty),\ j=1,2.$$
Moreover, in light of \cite[Lemma 2.1]{KY3} (see also \cite[Lemma 2.3]{KOSY})   the sequence
$$\sum_{k=1}^N\sum_{\ell=1}^{m_k}\frac{\left\langle \chi\eta_1,a_j\partial_{\nu}\phi_{k,\ell}^j\right\rangle_{L^2(\partial\Omega)}}{\lambda_n}\phi_{k,\ell}^j,\quad N\in\mathbb N,\ j=1,2, $$
converges in the sense of $L^2(\Omega)$. Combining these two estimates, we deduce that for all $t>\tau_2$, the sequence
$$\sum_{k=1}^N\sum_{\ell=1}^{m_k} \left(-\int_0^{\tau_2}(t-s)^{\alpha_j-1}E_{\alpha_j,\alpha_j}(-\lambda_k^j(t-s)^{\alpha_j})\psi_1(s)\left\langle \chi\eta_1,a_j\partial_{\nu}\phi_{k,\ell}^j\right\rangle_{L^2(\partial\Omega)}ds\right)\phi_{k,\ell}^j,\quad N\in\mathbb N$$
converges in the sense of $D(A_j)\subset H^2(\Omega)$. Using this result, we deduce that for all $t>\tau_2$, 
the sequence
$$\sum_{k=1}^N\sum_{\ell=1}^{m_k} \left(-\int_0^{\tau_2}(t-s)^{\alpha_j-1}E_{\alpha_j,\alpha_j}(-\lambda_n(t-s)^{\alpha_j})\psi_1(s)\left\langle \chi\eta_1,a_j\partial_{\nu}\phi_{k,\ell}^j\right\rangle_{L^2(\partial\Omega)}ds\right)\partial_\nu\phi_{k,\ell}^j,\quad N\in\mathbb N$$
converges in the sense of $L^2(\partial\Omega)$ and, for all $t\in(\tau_2,+\infty)$, we have
\bel{t111b}\partial_\nu v_1^j(t,\cdot)=\sum_{k=1}^\infty\sum_{\ell=1}^{m_k} \left(-\int_0^{\tau_2}(t-s)^{\alpha_j-1}E_{\alpha_j,\alpha_j}(-\lambda_k^j(t-s)^{\alpha_j})\psi_1(s)\left\langle \chi\eta_1,a_j\partial_{\nu}\phi_{k,\ell}^j\right\rangle_{L^2(\partial\Omega)}ds\right)\partial_\nu\phi_{k,\ell}^j.\ee
Applying the Lebesgue dominate convergence theorem, we deduce that, for all $t\in(\tau_2+1,+\infty)$ and for $j=1,2$, we have
\bel{t111c}\begin{aligned}&\partial_\nu v_1^j(t,\cdot)\\
&= -\int_0^{\tau_2} (t-s)^{\alpha_j-1}\left(\sum_{k=1}^\infty\sum_{\ell=1}^{m_k} E_{\alpha_j,\alpha_j}(-\lambda_k^j(t-s)^{\alpha_j})\psi_1(s)\left\langle \chi\eta_1,a_j\partial_{\nu}\phi_{k,\ell}^j\right\rangle_{L^2(\partial\Omega)}\partial_\nu\phi_{k,\ell}^j\right)ds.\end{aligned}\ee
Applying formula (1.143) page 34 of \cite{P} we deduce that there exists $C>0$ such that, for all $s\in(0,\tau_2)$, $t>\tau_2+1$ and $k\in\mathbb N$, we have
$$\begin{aligned}\abs{(t-s)^{\alpha_j-1}E_{\alpha_j,\alpha_j}(-\lambda_k^j(t-s)^{\alpha_j})+\frac{(t-s)^{-1-\alpha_j}}{\Gamma(-\alpha_j)(\lambda_k^j)^2}}&\leq C(t-s)^{\alpha_j-1}(\lambda_k^j(t-s)^{\alpha_j})^{-3}\\
\ &\leq C(\lambda_k^j)^{-3}t^{-1-2\alpha_j}.\end{aligned}$$
In this last identity, we assume that $\Gamma(-1)^{-1}=0$.
In addition, using the fact that
$$(t-s)^{-1-\alpha_j}=t^{-1-\alpha_j}+\underset{t\to+\infty}{\mathcal O}(t^{-1-2\alpha_j}),\quad s\in(0,\tau_2),$$
we deduce that there exists $C'$ independent of $t$, $s$, $k$, such that
$$\begin{aligned}&\abs{(t-s)^{\alpha_j-1}E_{\alpha_j,\alpha_j}(-\lambda_k^j(t-s)^{\alpha_j})+\frac{t^{-1-\alpha_j}}{\Gamma(-\alpha_j)(\lambda_k^j)^2}}\\
&\leq \abs{(t-s)^{\alpha_j-1}E_{\alpha_j,\alpha_j}(-\lambda_k^j(t-s)^{\alpha_j})+\frac{(t-s)^{-1-\alpha_j}}{\Gamma(-\alpha_j)(\lambda_k^j)^2}}+ \abs{\frac{t^{-1-\alpha_j}-(t-s)^{-1-\alpha_j}}{\Gamma(-\alpha_j)(\lambda_k^j)^2}}\\
&\leq C'(\lambda_k^j)^{-2}t^{-1-2\alpha_j}.\end{aligned}$$
Applying this estimate, we deduce that, for  all $t> \tau_2+1$, we have
$$\begin{aligned}&\int_0^{\tau_2} \norm{\sum_{k=1}^\infty\sum_{\ell=1}^{m_k} \left[(t-s)^{\alpha_j-1}E_{\alpha_j,\alpha_j}(-\lambda_k^j(t-s)^{\alpha_j})+\frac{t^{-1-\alpha_j}}{\Gamma(-\alpha_j)(\lambda_k^j)^2}\right]\psi_1(s)\left\langle \chi\eta_1,a_j\partial_{\nu}\phi_{k,\ell}^j\right\rangle_{L^2(\partial\Omega)}\phi_{k,\ell}^j}_{D(A_j)}ds\\
&\leq Ct^{-1-2\alpha_j}\norm{\psi_1}_{L^1(\R_+)}\left(\sum_{n=1}^\infty\sum_{k=1}^{m_\ell} \abs{\frac{\left\langle \chi\eta_1,a_j\partial_{\nu}\phi_{k,\ell}^j\right\rangle_{L^2(\partial\Omega)}}{\lambda_k^j}}^2\right)^{\frac{1}{2}}ds\\
&\leq Ct^{-1-2\alpha_j}\norm{\chi\eta_1}_{H^{\frac{1}{2}}(\partial\Omega)},\end{aligned}$$
where $C>0$ is a constant independent of $t$ that may change from line to line. Using this estimate and the continuity of the map $D(A_j)\ni v\mapsto\partial_\nu v\in L^2(\partial\Omega)$, we obtain 
$$\partial_\nu v_1^j(t,\cdot)
=\frac{t^{-1-\alpha_j}}{\Gamma(-\alpha_j)}\left(\int_0^{+\infty}\psi_1(s)ds\right)\left(\sum_{k=1}^\infty\sum_{\ell=1}^{m_k}\frac{\left\langle \chi\eta_1,a_j\partial_{\nu}\phi_{k,\ell}^j\right\rangle_{L^2(\partial\Omega)}}{(\lambda_k^j)^2}\partial_{\nu}\phi_{n,k}^j\right)+\underset{t\to+\infty}{\mathcal O}(t^{-1-2\alpha_j}). $$
In this last identity $\mathcal O$ is considered with respect to the norm of $L^2(\partial\Omega)$.   Let us  consider $G_j\in H^2(\Omega)$ the solution of 
\bel{eqqq2}
\left\{ \begin{array}{rcll} 
-\textrm{div}(a_j\nabla_x G_j)+q_j G_j & = & 0, & x\in \Omega,\\
 G_j(x) & = &  \chi\eta_1(x), & x\in\pd \tilde{\Omega},\\
 G_j(x) & = &  0, & x\in\pd \omega_1.
\end{array}
\right.
\ee
Applying \cite[Lemma 2.1]{KY3}, we deduce that
$$\left\langle G_j,\phi_{k,\ell}^j\right\rangle_{L^2(\Omega;\rho_j dx)}=-\frac{\left\langle \chi\eta_1,a_j\partial_{\nu}\phi_{k,\ell}^j\right\rangle_{L^2(\partial\Omega)}}{\lambda_k^j}$$ 
and, using this identity, we obtain the following asymptotic  property
\bel{t111e}\partial_\nu v_1^j(t,\cdot)=-\frac{t^{-1-\alpha_j}}{\Gamma(-\alpha_j)}\left(\int_0^{+\infty}\psi_1(s)ds\right)\partial_\nu w_j+\underset{t\to+\infty}{\mathcal O}(t^{-1-2\alpha_j}),\ee
where $w_j=A_j^{-1}G_j$ with $G_j$ the solution of \eqref{eqqq2}. Combining this asymptotic property of $\partial_\nu v_1^j(t,\cdot)$ as $t\to+\infty$ with condition \eqref{t1e}, with $k=1$, we will prove by contradiction that $\alpha_1=\alpha_2$.

Let us assume that $\alpha_1\neq\alpha_2$. From now on, without loss of generality we assume that $\alpha_1<\alpha_2$. Notice that \eqref{t1e}, for $k=1$, implies that
$$\pm a_1\partial_\nu v_1^1(t,x)=\pm a_2\partial_\nu v_1^2(t,x),\quad (t,x)\in(0,+\infty)\times\Gamma_{out}.$$
Combining this identity with the fact that $\chi\eta_1$ is of constant sign, by eventually replacing $\eta_1$ by $-\eta_1$,  we can  assume that the function $\chi\eta_1$  is non-positive.    From \eqref{t111e} and \eqref{t1e}, when $k=1$, we deduce that, for a.e. $x\in\Gamma_{out}$, we have
$$\begin{aligned}&-\frac{t^{-1-\alpha_1}}{\Gamma(-\alpha_1)}\left(\int_0^{+\infty}\psi_1(s)ds\right)a_1\partial_\nu w_1(x) +\underset{t\to+\infty}{\mathcal O}(t^{-1-2\alpha_1})\\
&=-\frac{t^{-1-\alpha_2}}{\Gamma(-\alpha_2)}\left(\int_0^{+\infty}\psi_1(s)ds\right)a_2\partial_\nu w_2(x)+\underset{t\to+\infty}{\mathcal O}(t^{-1-2\alpha_2}).\end{aligned}$$
Since $\psi_1\geq0$, and $\psi_1\not\equiv0$, we deduce that 
$$\int_0^{+\infty}\psi_1(s)ds>0$$
and, for a.e. $x\in\Gamma_{out}$, we obtain
\bel{t111f}\frac{t^{-1-\alpha_1}}{\Gamma(-\alpha_1)}a_1\partial_\nu w_1(x) +\underset{t\to+\infty}{\mathcal O}(t^{-1-2\alpha_1})
=\frac{t^{-1-\alpha_2}}{\Gamma(-\alpha_2)}a_2\partial_\nu w_2(x)+\underset{t\to+\infty}{\mathcal O}(t^{-1-2\alpha_2}).\ee
Using the fact that $\chi\eta_1\in W^{2-\frac{1}{r},r}(\partial\Omega)$ and applying \cite[Theorem 2.4.2.5]{Gr}, we obtain $G_j\in W^{2,r}(\Omega)$. Then, since $r>\frac{d}{2}$, the Sobolev embedding theorem implies that $G_j\in \mathcal C(\overline{\Omega})$.  Therefore, applying again \cite[Theorem 2.4.2.5]{Gr}, we deduce that $w_j\in W^{2,d+1}(\Omega)\subset\mathcal C^1(\overline{\Omega})$.
Since $\chi\eta_1\leq0$ and $\chi\eta_1\not\equiv0$, the maximum principle (see e.g. \cite[Corollary 3.2]{GT}) implies that, for $j=1,2$, $G_j\leq0$ and $G_j\not\equiv0$. Moreover, using the fact that $-\textrm{div}(a_j\nabla_x w_j)+q_jw_j=\rho_j G_j\leq 0$ and $w_j|_{\partial\Omega}=0$, the strong maximum principle (see e.g. \cite[Theorem 3.5]{GT}) implies that 
$$w_j(x)<0,\quad  x\in\Omega.$$
Thus, the Hopf lemma (see \cite[Lemma 3.4]{GT}) implies that $$\partial_\nu w_j(x)>0,\quad x\in\partial\Omega,\ j=1,2.$$ In particular, we have $\norm{a_j\partial_\nu w_j}_{L^2(\Gamma_{out})}>0$, $j=1,2$. Taking the norm $L^2(\Gamma_{out})$ on both sides of \eqref{t111f}, we get
\bel{t111g}\frac{t^{-1-\alpha_1}}{|\Gamma(-\alpha_1)|}\norm{a_1\partial_\nu w_1}_{L^2(\Gamma_{out})}
\leq\underset{t\to+\infty}{\mathcal O}(t^{-1-2\alpha_1})+\frac{t^{-1-\alpha_2}}{|\Gamma(-\alpha_2)|}\norm{a_2\partial_\nu w_2}_{L^2(\Gamma_{out})}+\underset{t\to+\infty}{\mathcal O}(t^{-1-2\alpha_2}),\ee
where this time $\mathcal O$ is considered in term of functions taking values in $\R$.
Assuming that $\alpha_j\neq1$, multiplying this expression by $|\Gamma(-\alpha_1)|t^{1+\alpha_1}$ and sending $t\to+\infty$, we get
$$\norm{a_1\partial_\nu w_1}_{L^2(\Gamma_{out})}\leq0.$$
This contradicts  the fact that $\norm{a_1\partial_\nu w_1}_{L^2(\Gamma_{out})}>0$  and it follows that $\alpha_1=\alpha_2$.
 On the other hand, if $\alpha_1=1$, combining the fact that
$$t^{-1-\alpha_2}=\underset{t\to+\infty}{o}(t^{-3}),$$
with \eqref{t111g} and the fact that $\norm{a_2\partial_\nu w_2}_{L^2(\Gamma_{out})}>0$, we deduce  that $\Gamma(-\alpha_2)^{-1}=0$ which implies that $\alpha_2=1$. In the same way, if $\alpha_2=1$ one can check that $\alpha_1=1$. This proves that in all case $\alpha_1=\alpha_2$

\textbf{Step 4.} From now on we fix $\alpha_1=\alpha_2=\alpha$ and  $\Omega=\Omega_1=\Omega_2=\tilde{\Omega}\setminus\overline{\omega_1}$ which, according to our assumptions, is a $\mathcal C^2$ open connected set. Note that, in this context, the fact that one of the conditions (i), (ii), (iii) is fulfilled implies that one of the following conditions
\bel{cond1}(i')\ \rho_1=\rho_2,\quad (ii')\ a_1=a_2,\quad (iii')\ q_1=q_2\ee
is fulfilled. In this step, we will show  that condition \eqref{t1e}, for $k\geq1$, implies that
\bel{t11e} \rho_1=\rho_2,\quad a_1=a_2,\quad q_1=q_2.\ee
For this purpose, we use the notation of the third step. Repeating the arguments used in Step 1, 2, 3 and 4 in the proof of \cite[Theorem 2.2]{KLLY}, we deduce that the condition  \eqref{t1e} for $k\in\mathbb N$
implies that the following conditions
\bel{t11f}\lambda^1_k=\lambda^2_k,\quad m^1_k=m_k^2,\quad k\in\mathbb N,\ee
\bel{t11g}\partial_\nu\phi^1_{k,\ell}(x)=\partial_\nu\phi^2_{k,\ell}(x),\quad k\in\mathbb N,\ \ell=1,\ldots,m^1_k,\ x\in\partial\tilde{\Omega}\ee
are fulfilled.
Combining \eqref{t11f}-\eqref{t11g} with \eqref{ob1}-\eqref{ob2}, we deduce that, for all $k\in\mathbb N$, $\ell=1,\ldots,m^1_k$, $\phi_{k,\ell}=\phi^1_{k,\ell}-\phi^2_{k,\ell}$ satisfies
\[\left\{\begin{aligned}-\textrm{div}\left(a_1 \nabla_x \phi_{k,\ell}\right)+q_1\phi_{k,\ell} -\lambda^1_k\rho_1\phi_{k,\ell}&=0\quad &\textrm{in }\tilde{O},\\   \phi_{k,\ell}=\partial_\nu \phi_{k,\ell}&=0,\quad &\textrm{on }\Gamma_{out}\cap \partial\tilde{O}.\end{aligned}\right.\]
Therefore, applying results of unique continuation for elliptic equations, we deduce that
$$\phi^1_{k,\ell}(x)-\phi^2_{k,\ell}(x)=\phi_{k,\ell}(x)=0,\quad k\in\mathbb N,\ \ell=1,\ldots,m^1_k,\  x\in\tilde{O}.$$
Combining this with the fact that $\partial\omega_1=\partial(\omega_1\cup\omega_2)\subset \partial\tilde{O}$, we deduce that
$$\partial_\nu\phi^1_{k,\ell}(x)=\partial_\nu\phi^2_{k,\ell}(x),\quad k\in\mathbb N,\ \ell=1,\ldots,m^1_k,\ x\in\partial\omega_1.$$
This last identity and \eqref{t11g} imply
$$\partial_\nu\phi^1_{k,\ell}(x)=\partial_\nu\phi^2_{k,\ell}(x),\quad k\in\mathbb N,\ \ell=1,\ldots,m^1_k,\ x\in\partial\Omega.$$
Combining this with \eqref{t11f} and  the fact that one of the conditions \eqref{cond1} is fulfilled, we are in position to apply the inverse spectral result of \cite[Corollaries 1.5–1.7]{CK2} in order to deduce that \eqref{t11e} holds true.

\textbf{Step 5.} In this last step we will complete the proof of the theorem by proving that condition \eqref{t1e} with $k=0$ implies
 that \bel{tt1a}u_0^1=u_0^2,\quad  f_1=f_2.\ee
Using the fact that $\Omega_1=\Omega_2=\Omega$, $\alpha_1=\alpha_2=\alpha$ and the fact that \eqref{t11e} is fulfilled, we deduce that, for $j=1,2$, $v_0^j$ solves the problem
$$\left\{ \begin{array}{lll} 
&(\rho_1(x)\partial_t^{\alpha}v_0^j -\textrm{div}\left(a_1(x) \nabla_x v_0^j \right)+q_1(x) v_0^j)(t,x)  =  F_j(t,x), & (t,x)\in  (0,+\infty) \times \Omega,\\
&v_0^j(t,x)  =  0, & (t,x) \in  (0,+\infty) \times \partial\Omega, \\  
&\pd_t^\ell v_0^j(0,\cdot)  =  u_\ell^j, & \mbox{in}\ \Omega,\ \ell=0,...,\lceil\alpha\rceil-1.
\end{array}
\right.$$
Fixing $F=F_1-F_2$, $u_0=u_0^1-u_0^2$, $u_1\equiv0$, we deduce that $v_0=v_0^1-v_0^2$ solves
\bel{eq5}\left\{ \begin{array}{lll} 
&(\rho_1(x)\partial_t^{\alpha}v_0 -\textrm{div}\left(a_1(x) \nabla_x v_0 \right)+q_1(x) v_0)(t,x)  =  F(t,x), & (t,x)\in  (0,+\infty) \times \Omega,\\
&v_0(t,x)  =  0, & (t,x) \in  (0,+\infty) \times \partial\Omega, \\  
&\pd_t^\ell v_0(0,\cdot)  =  u_\ell, & \mbox{in}\ \Omega,\ \ell=0,...,\lceil\alpha\rceil-1.
\end{array}
\right.\ee
Moreover,  condition \eqref{t1e} for $k=0$ implies that
$$\partial_\nu v_0(t,x)=0,\quad (t,x)\in(0,+\infty)\times\Gamma_{out}.$$
Without loss of generality and by eventually extending $\Omega$ into a larger connected open set, we may assume that
\bel{t1g} v_0(t,x)=0,\quad (t,x)\in(0,+\infty)\times\Omega'\ee
for $\Omega'$ an open subset of $\Omega$.  We will give the proof of this result both in the case where condition (v) and (vi) are fulfilled. Indeed, assuming that (iv) is fulfilled, one can deduce \eqref{tt1a} from \eqref{t1e} with $k=0$ by applying \cite[Theorem 2.5]{JLLY}.

Let us first assume that condition (v) of Theorem \ref{t1} is fulfilled. Recall that since $\Omega_1=\Omega_2=\Omega$, condition (v) implies that $u_0^1=u_0^2$. Then, $v=v_0^1-v_0^2$ solves \eqref{eq5} with $u_\ell\equiv0$,   $\ell=0,...,\lceil\alpha\rceil-1$ and $F(t,x)=\sigma(t) (f_1(x)-f_2(x))$. For all $p>0$ and for $V(p,\cdot)$ the Laplace transform in time of $v_0$ at $p$,   the conditions
$$\left\{ \begin{array}{lll} 
&\rho_1(x)p^{\alpha}V_0(p,x) -\textrm{div}\left(a_1(x) \nabla_x V_0 \right)(p,x)+q_1(x) V_0(p,x)  =  f(x)\int_0^Te^{-pt}\sigma(t)dt, & x\in   \Omega,\\
&V_0(p,x)  =  0, & x\in \pd \Omega, \\  
&V_0(p,x)=0, & x\in\Omega'
\end{array}
\right.$$
are fulfilled, with $f=f_1-f_2$. Since $\sigma\not\equiv0$ by the uniqueness and the analiticity of the Laplace transform in time of $\sigma$ extended by zero to $(0,+\infty)$, there exists $0<r_1<r_2$ such that
$$\int_0^Te^{-pt}\sigma(t)dt\neq0,\quad p\in (r_1,r_2).$$
Thus, fixing 
$$W(p,\cdot)=\frac{V_0(p^{\frac{1}{\alpha}},\cdot)}{\int_0^Te^{-p^{\frac{1}{\alpha}}t}\sigma(t)dt},\quad p\in (r_1^\alpha,r_2^\alpha),$$
we deduce that $W(p,\cdot)$ satisfies, for all $p\in (r_1^\alpha,r_2^\alpha)$, the conditions
$$\left\{ \begin{array}{lll} 
&\rho_1(x)pW(p,x) -\textrm{div}\left(a_1(x) \nabla_x W \right)(p,x)+q_1(x) W(p,x)  =  f(x), & x\in  \Omega,\\
&W(p,x)  =  0, & x\in \pd \Omega, \\  
&W(p,x)=0, & x\in\Omega'.
\end{array}
\right.$$
On the other hand, repeating the above arguments we deduce that for $w\in L^2(0,+\infty;H^1(\Omega))$ the solution of the parabolic problem
$$\left\{ \begin{array}{lll} 
&(\rho_1(x)\partial_tw -\textrm{div}\left(a_1(x) \nabla_x w\right)+q_1(x) w)(t,x)  =  0, & (t,x)\in  (0,+\infty) \times \Omega,\\
&w(t,x)  =  0, & (t,x) \in  (0,+\infty) \times \pd \Omega, \\  
&w(0,\cdot)  =  f, & \mbox{in}\ \Omega,
\end{array}
\right.$$
 $W(p,\cdot)$ coincides with the Laplace transform in time of $w$ at $p>0$, denoted by $\hat{w}(p,\cdot)$. Moreover, $p\mapsto \hat{w}(p,\cdot)$ is analytic in $(0,+\infty)$ as a function taking values in $L^2(\Omega)$ and the condition
$$\hat{w}(p,x)=W(p,x)=0, \quad p\in (r_1^\alpha,r_2^\alpha),\ x\in\Omega'$$
implies that 
$$\hat{w}(p,x)=0, \quad p>0,\ x\in\Omega'.$$
By the uniqueness of the Laplace transform in time of $w|_{(0,+\infty)\times\Omega'}$, we get \bel{t1aa}w(t,x)=0,\quad (t,x)\in(0,+\infty)\times\Omega'.\ee
The unique continuation results for parabolic equations (e.g.  \cite[Theorem 1.1]{SS}) imply that $w\equiv0$ which implies that $f\equiv0$. Therefore, we have $f_1=f_2$ from which we get \eqref{tt1a}.

Finally, let us assume that condition (vi) is fulfilled. Consider the solution of the following initial boundary value problems
\begin{equation}\label{eq51}\begin{cases} 
(\rho_1(x)\partial_t^{\alpha}v_{0,1}^j -\textrm{div}\left(a_1(x) \nabla_x v_{0,1}^j \right)+q_1(x) v_{0,1}^j)(t,x)  =  0, & (t,x)\in  (0,+\infty) \times \Omega,\\
v_{0,1}^j(t,x)  =  0, & (t,x) \in  (0,+\infty) \times \partial\Omega, \\  
\begin{cases}
v_{0,1}^j=u_0^j & \mbox{if }0<\alpha\leq1,\\
v_{0,1}^j=u_0^j,\quad \partial_t v_{0,1}^j=0 & \mbox{if }1<\alpha<2
\end{cases} & \mbox{in }\{0\}\times \Omega,
\end{cases}
\end{equation}
\bel{eq52}\left\{ \begin{array}{lll} 
&(\rho_1(x)\partial_t^{\alpha}v_{0,2}^j -\textrm{div}\left(a_1(x) \nabla_x v_{0,2}^j \right)+q_1(x) v_{0,2}^j)(t,x)  =  \sigma(t)f_j(x), & (t,x)\in  (0,+\infty) \times \Omega,\\
&v_{0,2}^j(t,x)  =  0, & (t,x) \in  (0,+\infty) \times \partial\Omega, \\  
&\pd_t^\ell v_{0,2}^j(0,\cdot)  =  0, & \mbox{in}\ \Omega,\ \ell=0,...,\lceil\alpha\rceil-1.
\end{array}
\right.\ee
Note that $v_0^j=v_{0,1}^j+v_{0,2}^j$, $j=1,2$. Moreover, in view of condition (vi), the restriction of $v_{0,2}^j$ to $(0,\tau_0)\times\Omega$ solves
$$\left\{ \begin{array}{lll} 
&(\rho_1(x)\partial_t^{\alpha}v_{0,2}^j -\textrm{div}\left(a_1(x) \nabla_x v_{0,2}^j \right)+q_1(x) v_{0,2}^j)(t,x)  =  0, & (t,x)\in  (0,\tau_0) \times \Omega,\\
&v_{0,2}^j(t,x)  =  0, & (t,x) \in  (0,\tau_0) \times \partial\Omega, \\  
&\pd_t^\ell v_{0,2}^j(0,\cdot)  =  0, & \mbox{in}\ \Omega,\ \ell=0,...,\lceil\alpha\rceil-1.
\end{array}
\right.$$
Therefore, the uniqueness of this initial boundary value problem implies that $v_{0,2}^1=v_{0,2}^2=0$ on $(0,\tau_0)\times\Omega$. Therefore, we have $v_0^j=v_{0,1}^j$ on $(0,\tau_0)\times\Omega$, $j=1,2$. Thus, condition \eqref{t1e}, with $k=0$, implies that 
$$\partial_\nu v_{0,1}^1(t,x)=\partial_\nu v_{0,1}^2(t,x),\quad (t,x)\in(0,\tau_0)\times\Gamma_{out}.$$
Then, applying  \cite[Theorem 2.5]{JLLY}, we deduce that $u_0^1=u_0^2$. This implies that $v_{0,1}^1=v_{0,1}^2$ on $(0,+\infty)\times\Omega$. It follows that $v_{0}^j=v_{0,1}^1+v_{0,2}^j$, $j=1,2$, and condition \eqref{t1e}, with $k=0$, implies that
$$\partial_\nu v_{0,2}^1(t,x)=\partial_\nu v_{0,2}^2(t,x),\quad (t,x)\in(0,+\infty)\times\Gamma_{out}.$$
Repeating the arguments used for proving \eqref{tt1a} when (v) is fulfilled, we get $f_1=f_2$.
This proves that \eqref{tt1a} holds true and it completes the proof of the theorem.

\section{ Proof of Theorem \ref{tt2}} 

We fix $u^j$, $j=1,2$, the solution of \eqref{eq1} with $\Phi$ given by \eqref{g}, $(a,\rho,q)=(a_j,\rho_j,q_j)$ and $(u_0,F)=(u_0^j,F_j)$. According to Theorem \ref{t1}, the proof of the theorem will be completed if we show that, for any values of $T_0\in[\tau_2,T]$ and of $\delta\in(0,T_0-\tau_1)$, the condition \eqref{t6a} implies \eqref{t1d}. For this purpose, we fix $T_0\in[\tau_2,T]$, $\delta\in(0,T_0-\tau_1)$ and we assume that \eqref{t6a} is fulfilled.
We set $u=u^1-u^2$, where we recall that $u^j\in  W^{\lceil\alpha\rceil,1}(0,T;H^{\frac{7}{4}}(\Omega_j))\cap L^1(0,T;H^{2+\frac{7}{4}}(\Omega_j))$,  $j=1,2$$^1$\footnote{$^1$ See the discussion before the statement of Theorem \ref{tt2}.}. We remark that $u$ satisfies the following conditions
$$\left\{ \begin{array}{lll} 
&(\rho_1(x)\partial_t^{\alpha}u -\textrm{div}\left(a_1(x) \nabla_x u \right)+q_1(x) u)(t,x)  =  G(t,x)+F(t,x), & (t,x)\in  (0,T) \times (\Omega_1\cap\Omega_2),\\
&u(t,x)  =  0, & (t,x) \in  (0,T) \times \partial\tilde{\Omega}, \\  
&\pd_t^\ell u(0,\cdot)  =  u_\ell, & \mbox{in}\ \Omega_1\cap\Omega_2,\ \ell=0,...,\lceil\alpha\rceil-1.
\end{array}
\right.$$
In the above equation we set $F=F_1-F_2$, $u_0=u_0^1-u_0^2$, $u_1\equiv 0$ and
$$G=(\rho_2-\rho_1)\partial_t^\alpha u^2-\textrm{div}\left((a_2-a_1) \nabla_x u^2 \right)+(q_2-q_1) u^2\in L^1(0,T;H^s(\Omega_1\cap\Omega_2)).$$
Since supp$(\sigma)\subset(0,\tau_1)$ and $(0,\tau_1)\cap (T_0-\delta,T_0)=\emptyset$, we deduce that
$$ F(t,x)=0,\quad (t,x)\in (T_0-\delta,T_0)\times(\Omega_1\cap\Omega_2)$$
and it follows that
$$(\rho_1(x)\partial_t^{\alpha}u -\textrm{div}\left(a_1(x) \nabla_x u \right)+q_1(x) u)(t,x)  =  G(t,x),\quad (t,x)\in (T_0-\delta,T_0)\times(\Omega_1\cap\Omega_2).$$
Using the fact that $u^j\in   W^{\lceil\alpha\rceil,1}(0,T;H^{\frac{7}{4}}(\Omega_1\cap\Omega_2))\cap L^1(0,T;H^{2+\frac{7}{4}}(\Omega_1\cap\Omega_2))$,  $j=1,2$, we can apply the normal trace to the above equation in order to obtain
\bel{t6f}\partial_\nu (\rho_1\partial_t^{\alpha}u -\textrm{div}\left(a_1 \nabla_x u \right)+q_1 u)(t,x)=\partial_\nu G(t,x),\quad (t,x)\in (T_0-\delta,T_0)\times\Gamma_{out}.\ee
Since $u=u^1-u^2\in   W^{\lceil\alpha\rceil,1}(0,T;H^{\frac{7}{4}}(\Omega_1\cap\Omega_2))$, we deduce that $\partial_\nu u\in W^{\lceil\alpha\rceil,1}(0,T;L^2(\partial\tilde{\Omega}))$. Combining this with the fact that $u=0$ on $(0,T)\times\partial\tilde{\Omega}$, we deduce that
$$\partial_\nu (\rho_1\partial_t^{\alpha}u)(t,x)=\rho_1(x)\partial_t^{\alpha}\partial_\nu u(t,x),\quad (t,x)\in (T_0-\delta,T_0)\times\Gamma_{out}.$$
In the same way, condition \eqref{t6e} and \eqref{t6a}  imply that
$$\partial_\nu u(t,x)=\partial_\nu\left(\textrm{div}\left(a_1 \nabla_x u \right)\right)(t,x)=0,\quad (t,x)\in(T_0-\delta,T_0)\times\Gamma_{out}.$$ 
Thus, we find
$$\partial_\nu (q_1 u)(t,x)=u(t,x)\partial_\nu q_1(x)+q_1(x)\partial_\nu u(t,x)=0, \quad (t,x)\in (T_0-\delta,T_0)\times\Gamma_{out}.$$
It follows that
\bel{t6g}\partial_\nu (\rho_1\partial_t^{\alpha}u -\textrm{div}\left(a_1 \nabla_x u \right)+q_1 u)(t,x)=\rho_1(x)\partial_t^{\alpha}\partial_\nu u(t,x),\quad (t,x)\in (T_0-\delta,T_0)\times\Gamma_{out}.\ee
On the other hand, applying \eqref{t6d}-\eqref{t6e}, we deduce that
\bel{t6h}\begin{aligned}&\partial_\nu G(t,x)\\
&=\partial_\nu\left[(\rho_2-\rho_1)\partial_t^\alpha u^2-\left(\textrm{div}\left(a_2 \nabla_x u^2 \right)-\textrm{div}\left(a_1 \nabla_x u^2 \right)\right)+(q_2-q_1) u^2\right](t,x)\\
&=0,\quad (t,x)\in (T_0-\delta,T_0)\times\Gamma_{out}.\end{aligned}\ee
Combining this with \eqref{t6a} and \eqref{t6f}-\eqref{t6h}, we deduce that
\bel{tt6h}\partial_\nu u(t,x)=\partial_t^{\alpha}\partial_\nu u(t,x)=0,\quad (t,x)\in (T_0-\delta,T_0)\times\Gamma_{out}.\ee
Now let us fix $\phi\in \mathcal C^\infty_0(\Gamma_{out})$  and consider the function 
$$v_\phi(t):=\left\langle \partial_\nu  u(t,\cdot),\phi\right\rangle_{L^2(\partial\tilde{\Omega})}.$$
Using the fact that $\partial_\nu u\in W^{\lceil\alpha\rceil,1}(0,T;L^2(\partial\tilde{\Omega}))$, we deduce that $v_\phi\in W^{\lceil\alpha\rceil,1}(0,T)$ and condition \eqref{tt6h} implies
$$v_\phi(t)=\partial_t^{\alpha}v_\phi(t)=0,\quad t\in (T_0-\delta,T_0).$$
Thus, applying \cite[Theorem 1]{KJ}, we deduce that
$$\left\langle \partial_\nu  u(t,\cdot),\phi\right\rangle_{L^2(\partial\tilde{\Omega})}=v_\phi(t)=0,\quad t\in (0,T_0).$$
Since in the above identity $\phi\in \mathcal C^\infty_0(\Gamma_{out})$ is arbitrary chosen, we obtain
$$\partial_\nu u(t,x)=0,\quad (t,x)\in (0,T_0)\times\Gamma_{out}$$
which implies \eqref{t1d}, since $T_0\geq\tau_2$. Therefore, applying Theorem \ref{t1}, we deduce  that \eqref{t1aaa} is fulfilled.\qed

\section{Proof of Theorem \ref{t2}}  Repeating the arguments used in the proof of Theorem \ref{t1} combined with the time analyticity properties of solutions of \eqref{eq2}-\eqref{eq3} stated in Theorem \ref{t4}, we deduce that \eqref{t2a} implies that, for all $k\in\mathbb N\cup\{0\}$, we have
\bel{t2c}\partial_\nu v_k^{1}(t,x)=\partial_\nu v_k^{2}(t,x),\quad (t,x)\in(0,+\infty)\times \partial\tilde{\Omega},\ee
with $v_0^j$ the solution of \eqref{eq2} for $a\equiv1$, $q\equiv0$, $\omega=\omega_j$, $\Omega=\Omega_j$ and $(\alpha,\rho,B, u_0,u_{1},F)=(\alpha_j,\rho_j,B_j,u_0^j,0,F_j)$, $j=1,2$, and $v_k^{j}$, $j=1,2$, $k\in\mathbb N$, the solution of \eqref{eq3} for $a\equiv1$, $q\equiv0$ and $(\alpha,\rho,B)=(\alpha_j,\rho_j,B_j)$, $j=1,2$. We mention, that in the case $\omega_1=\omega_2=\emptyset$, the result of Theorem \ref{t2} can be deduced from \eqref{t2c} by combining the results of \cite[Theorem 2.3]{KLLY}  with some arguments similar to those used in Theorem \ref{t1}. For this purpose, from now on we  assume that $\alpha_1=\alpha_2=\alpha$ and we will prove that \eqref{t2c} implies \eqref{t2aa}. The proof of Theorem \ref{t2} will be decomposed into three steps.  First, applying \eqref{t2c} with  $k=1$ and exploiting condition \eqref{ob1}, \eqref{ob3}, we deduce that $\omega_1=\omega_2$. Then, applying \eqref{t2c} with $k\in\mathbb N$ we will deduce that
\bel{t2d}B_1=B_2,\quad \rho_1=\rho_2.\ee
Finally, combining all these results and applying \eqref{t2c} with $k=0$ we deduce that
\bel{t2e}u_0^1=u_0^2,\quad f_1=f_2.\ee

\textbf{Step 1.} In this step,  we will show that $\omega_1=\omega_2$. For this purpose let us assume the contrary.  For $p>p_0$ let us fix $V_1^j(p,\cdot)$ the Laplace transform in time of $v_1^j$ at $p$. From the definition of weak solution of \eqref{eq2} for $q\equiv0$, $a\equiv1$, $\omega=\omega_j$ and $(B,\rho, u_0,u_{1},F)=(B_j,\rho_j,u_0^j,0,F_j)$, $j=1,2$, we deduce that, for all $p>p_0$, $V_1^j(p,\cdot)$ solves
\bel{t11111a}\left\{\begin{aligned}-\Delta_x V_1^j(p,\cdot) +B_j\cdot\nabla_xV_1^j(p,\cdot) +\rho_jp^\alpha V_1^j(p,\cdot)&=0,\quad &\textrm{in }\Omega_j,\\   V_1^j(p,\cdot)&=\hat{\psi}_1(p)\chi\eta_1,\quad &\textrm{on }\partial\tilde{\Omega},\\ V_1^j(p,\cdot)&=0,\quad &\textrm{on }\partial\omega_j.\end{aligned}\right.\ee
Moreover, following the representation of solutions of problem \eqref{eq2} given in Theorem \ref{t5}, one can check that, for $s\in(3/2,2)$ and for all $p>p_0$, we have $t\mapsto e^{-pt}v_1^j\in L^1(\R_+;H^s(\Omega_j))$, $j=1,2$. Here $p_0>0$ can be chosen sufficiently large. Therefore, we can apply the Laplace transform in time to the identity \eqref{t2c}, with $k=1$, in order to get
\bel{t2f}\partial_\nu V_1^1(p,x)=\partial_\nu V_1^2(p,x),\quad p>p_0,\ x\in\partial\tilde{\Omega}.\ee
 Combining \eqref{t11111a}-\eqref{t2f}  with \eqref{ob1} and \eqref{ob3}, we deduce that the restriction of $V_1(p,\cdot)=V_1^1(p,\cdot)-V_1^2(p,\cdot)$ to $\tilde{O}$ satisfies, for all $p>p_0$, the conditions
\[\left\{\begin{aligned}-\Delta_x V_1(p,\cdot) +B_1\cdot\nabla_xV_1(p,\cdot) +\rho_1p^\alpha V_1(p,\cdot)&=0\quad &\textrm{in }\tilde{O},\\   V_1(p,\cdot)=\partial_\nu V_1(p,\cdot)&=0,\quad &\textrm{on }\partial\tilde{\Omega}\cap \partial\tilde{O}.\end{aligned}\right.\]
Therefore, applying results of unique continuation for elliptic equations, we deduce that
$$V_1(p,x)=0,\quad p>p_0,\ x\in\tilde{O}.$$
Using the fact that $\partial(\omega_1\cup\omega_2)\subset\partial\tilde{O}$, we deduce from this identity that
\bel{ppp} V_1^1(p,x)= V_1(p,x)=0,\quad p>p_0,\ x\in(\partial\omega_2)\setminus(\partial\omega_1).\ee
Moreover, since $\chi\eta_1\in W^{2-\frac{1}{r},r}(\partial\Omega)$  and $V_1^1(p,\cdot)$, $p>p_0$, solves \eqref{t11111a}, \cite[Theorem 2.4.2.5]{Gr} implies that $V_1^1(p,\cdot)\in W^{2,r}(\Omega_1)$. Therefore, fixing $\omega_*=\omega_2\setminus\overline{\omega_1}$ and repeating the arguments used in the proof of Step 2 of Theorem \ref{t1}, we deduce that, for all $p>p_0$, the restriction of $V_1^1(p,\cdot)$ to $\omega_*$ is lying in $H^1_0(\omega_*)$ and it  satisfies
\bel{t2g} -\Delta_x V_1(p,\cdot) +B_1\cdot\nabla_xV_1(p,\cdot) +\rho_1p^\alpha V_1(p,\cdot)=0\quad \textrm{in }\omega_*.\ee
It follows that for any $\phi\in\mathcal C^\infty_0(\omega_*)$ we have
$$\begin{aligned}0&=\left\langle -\Delta V_1(p,\cdot),\phi\right\rangle_{D'(\omega_*),\mathcal C^\infty_0(\omega_*)} +\left\langle B_1\cdot\nabla_xV_1(p,\cdot) +\rho_1p^\alpha V_1(p,\cdot),\phi\right\rangle_{D'(\omega_*),\mathcal C^\infty_0(\omega_*)}\\
&=\left\langle \nabla_x V_1(p,\cdot),\nabla_x\phi\right\rangle_{D'(\omega_*)^d,\mathcal C^\infty_0(\omega_*)^d} +\left\langle B_1\cdot\nabla_xV_1(p,\cdot) +\rho_1p^\alpha V_1(p,\cdot),\phi\right\rangle_{D'(\omega_*),\mathcal C^\infty_0(\omega_*)}\\
&=\left\langle \nabla_x V_1(p,\cdot),\nabla_x\phi\right\rangle_{L^2(\omega_*)^d} +\left\langle B_1\cdot\nabla_xV_1(p,\cdot) +\rho_1p^\alpha V_1(p,\cdot),\phi\right\rangle_{L^2(\omega_*)}.\end{aligned}$$
By density, we can extend this identity to any $\phi\in H^1_0(\omega_*)$ and chosing $\phi=V_1(p,\cdot)$, we obtain
$$\norm{\nabla_x V_1(p,\cdot)}_{L^2(\omega_*)^d}^2 +\left\langle B_1\cdot\nabla_xV_1(p,\cdot) +\rho_1p^\alpha V_1(p,\cdot),V_1(p,\cdot)\right\rangle_{L^2(\omega_*)}=0,\quad p>p_0.$$
On the other hand, applying \eqref{eq-rho}, we obtain
$$\begin{aligned}0=&\norm{\nabla_x V_1(p,\cdot)}_{L^2(\omega_*)^d}^2 +\left\langle B_1\cdot\nabla_xV_1(p,\cdot) +\rho_1p^\alpha V_1(p,\cdot),V_1(p,\cdot)\right\rangle_{L^2(\omega_*)}\\
&\geq \frac{\norm{\nabla_x V_1(p,\cdot)}_{L^2(\omega_*)^d}^2}{2} +\left(p^\alpha \rho_0-\norm{ B_1}_{L^\infty(\omega_*)}^2\right)\norm{ V_1(p,\cdot)}_{L^2(\omega_*)}^2.\end{aligned}$$
Choosing $p>p_1:=p_0+\rho_0^{-\frac{1}{\alpha}}(\norm{ B_1}^2_{L^\infty(\omega_*)}+1)^{\frac{1}{\alpha}}$, we obtain
$$\norm{ V_1(p,\cdot)}_{L^2(\omega_*)}=0$$
 which implies that
$$V_1^1(p,x)=0,\quad p>p_1,\ x\in\omega_*.$$
Combining this with results of unique continuation for elliptic equations we deduce that, for all $p>p_1$, $V_1^1(p,\cdot)=0$ on $\Omega_1$. Moreover, using the fact that $\psi_1$ is non-negative and $\psi_1\not\equiv0$, we deduce that $\hat{\psi}_1(p_1+1)>0$ and it follows
$$\chi\eta_1(x)=\frac{V_1^1(p_1+1,x)}{\hat{\psi}_1(p_1+1)}=0,\quad x\in\partial\tilde{\Omega}.$$
This contradicts the fact that $\chi\eta_1\not\equiv0$. Therefore, we have $\omega_1=\omega_2$.

\textbf{Step 2.} From now on we assume that $\omega_1=\omega_2=\omega$ and we set $\Omega_1=\Omega_2=\Omega$. In this step, we will show that \eqref{t2d} is fulfilled.  For $j=1,2$ and $p>p_0$, consider the boundary value problem
\bel{t2i} \left\{\begin{aligned}-\Delta_x V_j(p,\cdot) +B_j\cdot\nabla_xV_j(p,\cdot) +\rho_jp^{\alpha} V_j(p,\cdot)&=0\quad &\textrm{in }\Omega,\\     V_j(p,\cdot)&=h,\quad &\textrm{on }\partial\tilde{\Omega},\\
     V_j(p,\cdot)&=0,\quad &\textrm{on }\partial\omega_1.\end{aligned}\right.\ee
In light of \cite[Theorem 8.3]{GT}, for $h\in H^{\frac{1}{2}}(\partial\tilde{\Omega})$  this problem admits a unique solution $V^j(p,\cdot)\in H^1(\Omega)$ and we can associate this problem with the partial Dirichlet-to-Neumann map
\[
\mathcal N_j(p):H^{\frac{1}{2}}(\partial\tilde{\Omega})\ni h\longmapsto\partial_\nu V_j(p,\cdot)_{|\partial\tilde{\Omega}}\in H^{-\frac{1}{2}}(\partial\tilde{\Omega}),\quad j=1,2,\ p>p_0.
\]
In a similar manner to the proof of \cite[Theorem 2.3]{KLLY}, we can prove that \eqref{t2c}, with $k\in\mathbb N$, imply
\begin{equation}\label{t2j}
\mathcal N_1(p)=\mathcal N_2(p),\quad p>p_0.
\end{equation}
Applying \cite[Theorem 1.9]{Sa} and condition \eqref{ob1}, \eqref{ob3}, we deduce from this identity that $B_1=B_2$ on $\partial\Omega$. We consider $U$ an open ball containing $\overline{\tilde{\Omega}}$ and we extend $B_j$ into a function still denoted by $B_j$ lying in $\mathcal C^\gamma(U;\R^d)$ with $B_1=B_2$ on $U\setminus\Omega$. Then, we consider the following boundary value problem 
$$\left\{\begin{aligned}-\Delta_x W_j(p,\cdot) +B_j\cdot\nabla_xW_j(p,\cdot) +\rho_j\mathds{1}_{\Omega}p^\alpha W_j(p,\cdot)&=0\quad &\textrm{in }U\setminus\omega_1,\\     W_j(p,\cdot)&=h,\quad &\textrm{on }\partial U,\\
 W_j(p,\cdot)&=0,\quad &\textrm{on }\partial\omega_1\end{aligned}\right.$$
and its associated partial Dirichlet-to-Neumann map
\[
\tilde{\mathcal N}_j(p):H^{\frac{1}{2}}(\partial U)\ni h\longmapsto\partial_\nu W_j(p,\cdot)_{|\partial U}\in H^{-\frac{1}{2}}(\partial U),\quad j=1,2,\ p>p_0.
\]
Following the proof of \cite[Lemma 6.2]{Po}, one can check that \eqref{t2j} implies that 
\begin{equation}\label{t2k}
\tilde{\mathcal N}_1(p)=\tilde{\mathcal N}_2(p),\quad p>p_0.
\end{equation}
Combining this with condition \eqref{ob3}, \cite[Proposition 2.1]{Po}, the proof of \cite[Theorem 1.1]{Po} and  density arguments comparable to the ones used in \cite[Theorem 1.3.]{Ki5} (see also \cite{AU}), we deduce that for $B=B_1-B_2$ extended by zero to $\R^3$ there exists $\phi\in W^{1,\infty}(\R^3)$ such that $B=\nabla_x \phi$ on $\R^3$ and 
\bel{t2l}|B_1|^2-\frac{\textrm{div}(B_1)}{2}+\rho_1p^\alpha=|B_2|^2-\frac{\textrm{div}(B_2)}{2}+\rho_2p^\alpha,\quad p>p_0.\ee
Using the fact that $B=0$ on  $(\R^3\setminus\overline{\tilde{\Omega}})\cup \overline{\tilde{O}}\cup\omega_1$ which is connected, by subtracting a constant to $\phi$ we may assume that $\phi=0$ on $(\R^3\setminus\tilde{\Omega})\cup \overline{\tilde{O}}\cup\omega_1$. Combining this with the fact that $\partial\Omega=\partial\tilde{\Omega}\cup\partial\omega_1$, we obtain $\phi\in H^1_0(\Omega)$. On the other hand, we get from \eqref{t2l} that $\phi$ satisfies
$$-\Delta\phi +2(B_1+B_2)\cdot\nabla\phi=2p^\alpha (\rho_2-\rho_1),\quad p>p_0.$$
Dividing this expression by $p^\alpha$ and sending $p\to+\infty$, we find $\rho_1=\rho_2$. Then, it follows that $\phi\in H^1_0(\Omega)$ satisfies $-\Delta\phi +2(B_1+B_2)\cdot\nabla\phi=0$ on $\Omega$ which combined with  \cite[Theorem 8.3]{GT} implies that $\phi\equiv 0$. Therefore, we have $B_1=B_2$ which implies \eqref{t2d}. 

\textbf{Step 3.} In this step we will complete the proof of the theorem by showing that \eqref{t2c} and \eqref{t2d} imply \eqref{t2e}. For this purpose, we fix  $v_0=v_0^1-v_0^2$, the condition \eqref{t2c} for $k=0$ implies that
$$\partial_\nu v_0(t,x)=0,\quad (t,x)\in(0,+\infty)\times\partial\tilde{\Omega}.$$
Without loss of generality and by eventually extending $\Omega$ into a larger connected open set, we may assume that \eqref{t1g} is fulfilled.
We will complete the proof of the theorem by showing  that \eqref{t2c} implies that \eqref{tt1a} holds true. Like in Theorem \ref{t1}, we will give the proof of this result both in the case where condition (v) and (vi) are fulfilled and we refer to \cite[Theorem 2.5]{JLLY} for the proof of this result when (iv) is fulfilled.

Let us first assume that condition (v) of Theorem \ref{t1} is fulfilled.  In a similar way to Step 5 of Theorem \ref{t1}, we can find $p_0<r_1<r_2$ such that
$$\int_0^Te^{-pt}\sigma(t)dt\neq0,\quad p\in (r_1,r_2).$$
Here $p_0>0$ corresponds to the value appearing in the Definition \ref{d1} of weak solution of \eqref{eq0}. Without loss of generality, we refer to $p_0$ as the maximum of the  value appearing in the Definition \ref{d1} for solutions of problem \eqref{eq0} for $\alpha=1$ and for the  value $\alpha$ of Theorem \ref{t2}.
Then, in a similar way to the proof of Theorem \ref{t1}, fixing  $V_0(p,\cdot)$, $p>p_0$, the Laplace transform in time of $v_0$, and $$W(p,\cdot)=\frac{V_0(p^{\frac{1}{\alpha}},\cdot)}{\int_0^Te^{-p^{\frac{1}{\alpha}}t}\sigma(t)dt},\quad p\in (r_1^\alpha,r_2^\alpha),$$
we deduce that $W(p,\cdot)$ satisfies, for all $p\in (r_1^\alpha,r_2^\alpha)$, the conditions
$$\left\{ \begin{array}{lll} 
&\rho_1(x)pW(p,x) -\Delta W(p,x)+B_1\cdot\nabla_x W(p,x)  =  f(x), & x\in  \Omega,\\
&W(p,x)  =  0, & x\in \pd \Omega, \\  
&W(p,x)=0, & x\in\Omega'.
\end{array}
\right.$$
Then, in a similar way to the last step of the proof of Theorem \ref{t1}, for $w\in L^2_{loc}(0,+\infty;H^1(\Omega))$ the solution of the parabolic problem
$$\left\{ \begin{array}{lll} 
&(\rho_1(x)\partial_tw -\Delta w+B_1\cdot\nabla_x w)(t,x)  =  0, & (t,x)\in  (0,+\infty) \times \Omega,\\
&w(t,x)  =  0, & (t,x) \in  (0,+\infty) \times \partial\tilde{\Omega}, \\  
&w(0,\cdot)  =  f, & \mbox{in}\ \Omega,
\end{array}
\right.$$
 we deduce that $W(p,\cdot)$ coincides with the Laplace transform in time of $w$ at $p>p_0$, denoted by $\hat{w}(p,\cdot)$.
Then, from the fact  that 
$$V_0(p,x)=0,\quad  x\in\Omega',\ p>p_0$$
and the analyticity of $p\mapsto\hat{w}(p,\cdot)$, $p>p_0$, we deduce that
$$w(t,x)=0,\quad (t,x)\in(0,+\infty)\times\Omega'.$$
Combining this with a unique continuation argument similar to the one used in the proof of Theorem \ref{t1},  we deduce that $f_1=f_2$. This proves that \eqref{tt1a} holds true when condition (v) is fulfilled. In the same way, assuming that condition (vi) is fulfilled, we deduce that \eqref{t2e} implies that $u_0^1=u_0^2$, $f_1=f_2$ and that \eqref{tt1a} holds true. This completes the proof of the theorem.\qed

\section{Proof of Corollary \ref{c1}, \ref{cc1} and \ref{c2}}

This section is devoted to the proof of Corollary \ref{c1}, \ref{cc1}. We will omit the proof of  Corollary \ref{c2} which can be deduced from some arguments used in Corollary \ref{c1} and the arguments used in \cite[Corollary 2.7]{KLLY}. We start with Corollary \ref{c1}.

\ \\
\textbf{Proof of Corollary \ref{c1}.} Let us first consider, for $k\in\mathbb N$, the initial boundary value problems
\begin{equation}\label{eq111}
\begin{cases}
\partial_t^{\alpha_j}v_k^j -\Delta_{g_j}v_k^j+q_j(x) v_k^j =  0, & \mbox{in }(0,+\infty)\times M_j,\\
v_k^j(t,x)= d_k\psi_k(t)\chi(x) \eta_k(x), & (t,x)\in (0,+\infty)\times\partial M_1, \\
\begin{cases}
v_k^j=0 & \mbox{if }0<\alpha\leq1,\\
v_k^j=\partial_t v_k^j=0 & \mbox{if }1<\alpha<2,
\end{cases} & \mbox{in }\{0\}\times M_j,
\end{cases}
\end{equation} 
\begin{equation}\label{eq112}
\begin{cases}
\partial_t^{\alpha_j}v_0^j -\Delta_{g_j}v_0^j+q_j(x) v_0^j =  \sigma(t)f_j(x), & \mbox{in }(0,+\infty)\times M_j,\\
v_0^j(t,x)= 0 & (t,x)\in (0,+\infty)\times\partial M_1, \\
\begin{cases}
v_0^j=u_0^j & \mbox{if }0<\alpha\leq1,\\
v_0^j=u_0^j,\ \partial_t v_0^j=0 & \mbox{if }1<\alpha<2,
\end{cases} & \mbox{in }\{0\}\times M_j.
\end{cases}
\end{equation} 
Following the argumentation of Theorem \ref{t1}, we deduce that the condition \eqref{t4d} implies that, for all $k\in\mathbb N\cup\{0\}$, we have
\bel{c1d}\partial_{\nu}v_k^1(t,x)=\partial_{\nu}v_k^2(t,x),\quad (t,x)\in(0,+\infty)\times\Gamma_{out}.\ee
In a similar way to Theorem \ref{t1}, we will show that condition \eqref{c1d}, with $k\in\mathbb N$, implies that $(M_1,g_1)$ and $(M_2,g_2)$ are isometric and \eqref{mu_c} holds true. For this purpose, let us start by proving that the condition \eqref{c1d} with $k=1$ implies that $\alpha_1=\alpha_2$. To do so, we will proceed in a similar way to the Step 3 in the proof of Theorem \ref{t1}. Let us consider the operator $A_j$, $j=1,2$, acting on $L^2(M_j)$ with domain $D(A_j)=H^2(M_j)\cap H^1_0(M_j)$ defined  by
\[
A_j w:=-\Delta_{g_j}w+q_j w,\quad w\in D(A_j).
\]
We fix $\{\lambda^j_k\}_{k\in\mathbb N}$ and $m^j_k\in\mathbb N$  the strictly increasing sequence of the eigenvalues of $A_j$ and the algebraic multiplicity of $\lambda^j_k$, respectively. For each eigenvalue $\lambda^j_k$, we introduce a family $\{\phi^j_{k,\ell}\}_{\ell=1}^{m^j_k}$ of eigenfunctions of $A_j$, i.e.,
\[
A_j\phi^j_{k,\ell}=\lambda^j_k\phi^j_{k,\ell},\quad\ell=1,\ldots,m^j_k,
\]
which forms an orthonormal basis in $L^2(M_j)$ of the algebraic eigenspace of $A_j$ associated with $\lambda^j_k$. Following  the arguments used in the Step 3 of the proof of Theorem \ref{t1}, one can check that we have
\bel{c1e}\partial_\nu v_1^j(t,\cdot)=-\frac{t^{-1-\alpha_j}}{\Gamma(-\alpha_j)}\left(\int_0^{+\infty}\psi_1(s)ds\right)\partial_\nu w_j+\underset{t\to+\infty}{\mathcal O}(t^{-1-2\alpha_j}),\ee
where $w_j=A_j^{-1}G_j$ with $G_j$ the solution of 
\bel{c1f}
\left\{ \begin{array}{rcll} 
-\Delta_{g_j} G_j+q_j G_j & = & 0, & \textrm{in } M_j,\\
 G_j(x) & = &  \chi\eta_1(x), & x\in\pd M_j.
\end{array}
\right.
\ee
Combining the asymptotic property \eqref{c1e} of $\partial_\nu v_1^j(t,\cdot)$ as $t\to+\infty$ with condition \eqref{c1d}, with $k=1$, we will prove by contradiction that $\alpha_1=\alpha_2$.

Let us assume that $\alpha_1\neq\alpha_2$. From now on, without loss of generality we assume that $\alpha_1<\alpha_2$. In a similar way to Theorem \ref{t1}, without loss of generality  we can  assume that the function $\chi\eta_1$  is non-positive.  Since $\chi\eta_1 \in \mathcal C^3(\partial M_1)$, we deduce that $G_j\in \mathcal C^2(M_j)$. Since $\chi\eta_1\leq0$ and $\chi\eta_1\not\equiv0$, the maximum principle  stated on the manifold $M_j$ (see e.g. \cite[Theorem 9.3]{PS}) implies that, for $j=1,2$, $G_j\leq0$ and $G_j\not\equiv0$. Moreover, using the fact that $-\Delta_{g_j} w_j+q_jw_j= G_j\leq 0$ and $w_j|_{\partial M_j}=0$, the strong maximum principle (see e.g. \cite[Theorem 9.3]{PS}) implies that 
$$w_j(x)<0,\quad  x\in M_j\setminus\partial M_j.$$
Thus, the Hopf lemma applied to the manifold $(M_j,g_j)$ (see \cite[Lemma 3.1.]{ACT}) implies that $$\partial_\nu w_j(x)>0,\quad x\in\partial M_1,\ j=1,2.$$ In particular, we have $\norm{\partial_\nu w_j}_{L^2(\Gamma_{out})}>0$. Taking the norm $L^2(\Gamma_{out})$ on both sides of \eqref{c1d}, for $k=1$, and applying \eqref{c1e}, we get
\bel{t111gg}\frac{t^{-1-\alpha_1}}{|\Gamma(-\alpha_1)|}\norm{\partial_\nu w_1}_{L^2(\Gamma_{out})}\leq\underset{t\to+\infty}{\mathcal O}(t^{-1-2\alpha_1})+\frac{t^{-1-\alpha_2}}{|\Gamma(-\alpha_2)|}\norm{\partial_\nu w_2}_{L^2(\Gamma_{out})}+\underset{t\to+\infty}{\mathcal O}(t^{-1-2\alpha_2}).\ee
 Assuming that $\alpha_j\neq1$, $j=1,2$, multiplying the expression \eqref{t111gg} by $|\Gamma(-\alpha_1)|t^{1+\alpha_1}$ and sending $t\to+\infty$, we get
$$\norm{\partial_\nu w_1}_{L^2(\Gamma_{out})}\leq0.$$
This contradicts  the fact that $\norm{\partial_\nu w_j}_{L^2(\Gamma_{out})}>0$, $j=1,2$, and it follows that $\alpha_1=\alpha_2$.
Using this result and repeating the arguments used at the end of the proof of Step 3 of Theorem \ref{t1}, we deduce that in all  case $\alpha_1=\alpha_2$. In the same way, following the proof of \cite[Corollary 2.4]{KLLY}, one can check that condition \eqref{c1d}, for $k\in\mathbb N$, implies that $(M_1,g_1)$ and $(M_2,g_2)$ are isometric and there exist $\phi\in C^\infty(M_2;M_1)$, an isomtery from $(M_2,g_2)$ to $(M_1,g_1)$, fixing $\partial M_1$ and depending only on $(M_j,g_j)$, $j=1,2$, such that  $q_2=q_1\circ\phi$. Therefore, fixing $\tilde{v}(t,x)=v_0^1(t,\phi(x))$, $(t,x)\in(0,+\infty)\times M_2$, we deduce that $\tilde{v}$ solves
$$\begin{cases}
\partial_t^{\alpha_1}\tilde{v}-\Delta_{g_2}\tilde{v}+q_2(x) \tilde{v} =  \sigma(t)f_1(\phi(x)), & \mbox{in }(0,+\infty)\times M_2,\\
\tilde{v}(t,x)= 0 & (t,x)\in (0,+\infty)\times\partial M_2, \\
\begin{cases}
\tilde{v}(0,x)=u_0^1(\phi(x)) & \mbox{if }0<\alpha\leq1,\\
\tilde{v}(0,x)=u_0^1(\phi(x)),\ \partial_t \tilde{v}(0,x)=0 & \mbox{if }1<\alpha<2,
\end{cases} & x\in M_2.
\end{cases}$$
Moreover, using the fact that $\phi$ fix the boundary $\partial M_1$, we get
$$\partial_\nu\tilde{v}(t,x)=\partial_\nu [v_0^1(t,\phi(x))]=\partial_\nu v_0^1(t,x),\quad (t,x)\in (0,+\infty)\times\Gamma_{out}.$$
Combining this with \eqref{c1d} for $k=0$ we deduce that
$$\partial_\nu\tilde{v}(t,x)=\partial_\nu v_0^2(t,x),\quad t\in(0,+\infty),\ x\in \Gamma_{out}$$
and we deduce that $v_0=\tilde{v}-v_0^2$ satisfies the following conditions
$$\begin{cases}
\partial_t^{\alpha_1}v_0-\Delta_{g_2}v_0+q_2(x) v_0 =  \sigma(t)(f_1(\phi(x))-f_2(x)), & \mbox{in }(0,+\infty)\times M_2,\\
v_0(t,x)= 0 & (t,x)\in (0,+\infty)\times\partial M_2, \\
\partial_\nu v_0(t,x)= 0 & (t,x)\in (0,+\infty)\times\Gamma_{out}, \\
\begin{cases}
v_0(0,x)=(u_0^1(\phi(x))-u_0^2(x)) & \mbox{if }0<\alpha\leq1,\\
v_0(0,x)=(u_0^1(\phi(x))-u_0^2(x)),\ \partial_t \tilde{v}(0,x)=0 & \mbox{if }1<\alpha<2,
\end{cases} & x\in M_2.
\end{cases}$$
Thus repeating the arguments used in the last step of the proof of Theorem \ref{t1}, we deduce from these conditions that
$$f_1(\phi(x))=f_2(x),\quad u_0^1(\phi(x))=u_0^2(x),\quad x\in M_2.$$
This completes the proof of the corollary.\qed

\ \\
\textbf{Proof of Corollary \ref{cc1}.}

Let $u^j$, $j=1,2$, be the solution of 

$$\begin{cases}
\partial_t^{\alpha}u^j -\Delta_{g_j} u^j +q_j(x) u^j =  F_j, & \mbox{in }(0,T)\times M,\\
u^j= \Phi, & \mbox{on } (0,T)\times\partial M, \\
\begin{cases}
u^j=u_0^j & \mbox{if }0<\alpha\leq1,\\
u^j=u_0^j,\quad \partial_t u_j=0 & \mbox{if }1<\alpha<2,
\end{cases} & \mbox{in }\{0\}\times M.
\end{cases}$$

 According to Corollary \ref{c1}, the proof of the theorem will be completed if we show that the conditions \eqref{cc1a}-\eqref{cc1b}, for some arbitrary chosen $T_0\in[\tau_2,T]$ and $\delta\in(0,T_0-\tau_1)$,  imply \eqref{t4d}. From now on we fix $T_0\in[\tau_2,T]$, $\delta\in(0,T_0-\tau_1)$ and we assume that the conditions \eqref{cc1a}-\eqref{cc1b} are fulfilled. We fix $u=u^1-u^2$, where we recall that $u^j\in  W^{\lceil\alpha\rceil,1}(0,T;H^{\frac{7}{4}}(M))\cap L^1(0,T;H^{2+\frac{7}{4}}(M))$,  $j=1,2$. We remark that $u$ satisfies the following conditions
$$\left\{ \begin{array}{lll} 
&\partial_t^{\alpha}u -\Delta_{g_1} u +q_1 u  =  G+F, & \textrm{in }  (0,T) \times M,\\
&u(t,x)  =  0, & (t,x) \in  (0,T) \times \partial M, \\  
&\pd_t^\ell u(0,\cdot)  =  u_\ell, & \mbox{in}\ M,\ \ell=0,...,\lceil\alpha\rceil-1.
\end{array}
\right.$$
In the above equation we set $F=F_1-F_2$, $u_0=u_0^1-u_0^2$, $u_1\equiv 0$ and
$$G=\Delta_{g_1} u_2-\Delta_{g_2}u_2 +(q_2-q_1) u_2\in L^1(0,T;H^s(M)),\quad  s\in(3/2,2).$$
Then, combining conditions \eqref{cc1c}-\eqref{cc1d} with the arguments used in Theorem \ref{tt2}, we find that
$$\partial_\nu u(t,x)=\partial_t^{\alpha}\partial_\nu u(t,x)=0,\quad (t,x)\in (T_0-\delta,T_0)\times\Gamma_{out}.$$
Therefore, repeating the arguments used at the end of the proof of Theorem \ref{tt2}, we obtain
$$\partial_\nu u(t,x)=0,\quad (t,x)\in (0,T_0)\times\Gamma_{out}$$
which implies \eqref{t4d}. Thus, following Corollary \ref{c1}, we deduce the results sated in Corollary \ref{cc1}.\qed

\section*{Acknowledgments}

 This work was supported by  the French National Research Agency ANR (project MultiOnde) grant ANR-17-CE40-0029.

\end{document}